\documentclass[11pt]{amsart}
\usepackage{amssymb,amscd,verbatim,graphicx}
\usepackage[T1]{fontenc}
\usepackage{mathrsfs}
\usepackage{xy}
\usepackage{bbm} 
\xyoption{all}
\usepackage{ulem}
\usepackage[colorlinks,linkcolor=blue,citecolor=blue,urlcolor=red]{hyperref}

\swapnumbers
\numberwithin{equation}{section}
\newtheorem{thm}{Theorem}[subsection]
\newtheorem{propose}[thm]{Proposition}
\newtheorem{lemma}[thm]{Lemma}
\newtheorem{cor}[thm]{Corollary}
\theoremstyle{definition}

\newtheorem{conj}[thm]{Period Conjecture}
\newtheorem{coreconj}[thm]{Core Equivalence Conjecture}
\newtheorem{remark}[thm]{Remark}

\newcommand{\uOm}{\underline{\Omega}}
\newcommand{\bOm}{{\bf \Omega}}

\newcommand{\M}{\mathcal{M}_1}        
\newcommand{\tM}{{}^t\!\mathcal{M}_{1}}

\newcommand{\dr}{\mathrm{dR}}

\renewcommand{\d}{{\text{\LARGE $\cdot $}}}

\newcommand{\Hom}{\operatorname{Hom}}      
\newcommand{\Ext}{\operatorname{Ext}}      
\newcommand{\DM}{{\operatorname{DM}}}  
\newcommand{\DA}{{\operatorname{DA}}}

\newcommand{\Div}{\operatorname{Div}}

\newcommand{\eff}{{\operatorname{eff}}}
\newcommand{\ihom}{{\rm\underline{Hom}}}  

\newcommand{\Mod}{\mathrm{Mod}_{\Z,K}}

\newcommand{\C}{\mathbb{C}}     

\newcommand{\Q}{\mathbb{Q}}     
\newcommand{\Z}{\mathbb{Z}}     

\newcommand{\G}{\mathbb{G}}     
\newcommand{\EExt}{{\rm \mathbb{E}xt}} 
\newcommand{\HH}{\mathbb{H}}     

\renewcommand{\ker}{\operatorname{Ker}}  
\newcommand{\coker}{\operatorname{Coker}} 
\newcommand{\Pic}{\operatorname{Pic}}     
\newcommand{\RPic}{\operatorname{RPic}}     
\newcommand{\Alb}{\operatorname{Alb}}     
\newcommand{\LAlb}{\operatorname{LAlb}}     

\newcommand{\LA}[1]{\mbox{${\rm L}_{#1}{\rm Alb}$}}
\newcommand{\RA}[1]{\mbox{${\rm R}^{#1}{\rm Pic}$}}

\renewcommand{\Mod}{\operatorname{Mod}}

\newcommand{\Tot}{\operatorname{Tot}}     
\newcommand{\NS}  {\operatorname{NS}}      
\newcommand{\tors}{{\operatorname{tors}}}

\newcommand{\tr}{{\operatorname{tr}}}        

\newcommand{\by}[1]{\stackrel{#1}{\rightarrow}}
\newcommand{\longby}[1]{\stackrel{#1}{\longrightarrow}}

\newcommand{\iso}{\stackrel{\sim}{\longrightarrow}}

\renewcommand{\tilde}{\widetilde}
\newcommand{\df}{\mbox{\,${:=}$}\,}
\newcommand{\ie}{{\it i.e.}, }
\newcommand{\cf}{{\it cf. }}
\newcommand{\eg}{{\it e.g. }}

\newcommand{\loccit}{{\it loc. cit. }}
\newcommand{\et} {{\rm \acute{e}t}}
\newcommand{\eh} {{\rm \acute{e}h}}

\newcommand{\Nis}{{\rm Nis}}
\newcommand{\cdh}{{\rm cdh}}
\newcommand{\an}{{\rm an}}

\newcommand{\fr}{{\rm fr}}

\newcommand{\gm}{{\rm gm}}

\renewcommand{\bar}{\overline}
\newcommand{\into}{\hookrightarrow}

\renewcommand{\lim}{\varprojlim}

\newcommand{\boxtensor}{\def\boxtimesten{\Box\kern-7.59pt\raise1.2pt
\hbox{$\times$} }}                                  

\newcounter{elno}                   

\newenvironment{points}{
                         \begin{list}{(\roman{elno})
                                     }{\usecounter{elno}}
                      }{
                         \end{list}}

\newcommand{\cE}{\mathcal{E}}

\newcommand{\cM}{\mathcal{M}}

\newcommand{\cU}{\mathcal{U}}

\newcommand{\cZ}{\mathcal{Z}}
\newcommand{\cMR}{\mathcal{MR}}

\renewcommand{\phi}{\varphi}
\renewcommand{\epsilon}{\varepsilon}
\renewcommand*{\d}{{\scalebox{0.5}{$\bullet$}}}
\begin{document}
\title{Homological periods and higher cycles}
\author[]{Luca Barbieri-Viale}
\address{Dipartimento di Matematica ``F. Enriques'', Universit{\`a} degli Studi di Milano\\ via C. Saldini, 50\\ Milano I-20133\\ Italy}
\email{Luca.Barbieri-Viale@unimi.it}

\begin{abstract}
For any scheme which is algebraic over a subfield of the complex numbers we here construct an homological regulator from Suslin homology to period homology and a higher cycle class map from Bloch's higher Chow group to the period Borel-Moore homology. Over algebraic numbers, making use of the motivic Albanese, we provide a purely geometric description of these period homologies in degree 1 and we characterise the $\Q/\Z$-cokernel of these regulators in terms of torsion zero-cycles, showing that Grothendieck period conjectures imply generalised Ro\u{\i}tman theorems.
\end{abstract}

\maketitle

\section{Introduction}
Let $X$ be a scheme which we assume to be separated and of finite type over $K$ a subfield of the complex numbers. Denote by $$H^{p,q}_\varpi (X)\df H^p(X_\an, \Z_\an (q))\cap H^{p}_\dr(X)$$ the arithmetic invariant introduced by Grothendieck (see \cite{GrdR} and \cf \cite{ABVB}) which is given by the subgroup $$H^{p,q}_\varpi (X)\subseteq H^{p, q}_\an(X) := H^p(X_\an, \Z_\an (q))$$ of those $p$-classes in the $q$-twisted singular cohomology groups of $X_\an$ which are landing in de Rham cohomology $H^p_\dr (X)$ under $$\varpi^{p, q}_X : H^p(X_\an, \Z_\an (q))\otimes_\Z \C\cong H^p_\dr (X)\otimes_K \C$$
the period isomorphism.
Consider the Suslin-Voevodsky motivic complex $\Z(q)$ (\eg see \cite[Def. 3.1]{VL}) and let $H^{p, q}(X)$ be the \'etale motivic cohomology; recall that $$H^{p,q} (X)\cong H^{p}_\eh(X, \Z (q))$$ is computed by the $\eh$-cohomology (see \cite[\S 10.2]{BVK} for the $\eh$-topology) and refer to \cite{ABVB} where we have constructed an integrally defined cohomological period regulator  
$$r^{p,q}_\varpi\colon H^{p,q} (X)\to H^{p,q}_\varpi (X).$$

\subsection{Cohomological period conjectures} A version of the generalized Grothendieck period conjecture for motivic cohomology is the expectation that $r^{p,q}_\varpi$ is surjective over $K = \bar \Q$ (see \cite[\S 1.3]{ABVB} and \cf \cite{Bo} and \cite{BC} for a classical formulation of period conjectures involving cycles; see \cite{HG} for an updated account including motivic periods). The main result of \cite{ABVB} is the geometric description of the Grothendieck arithmetic invariant $H^{1,q}_\varpi (X)$, proving the period conjecture, for $p=1$ and all twists $q$.

Moreover, in general, $r^{p,q}_\varpi$ is always surjective on torsion and it is surjective if and only if  $r^{p,q}_\varpi$ is surjective rationally (see \cite[Sect. 1]{ABVB} and \cite{RS} for details). By the change of topology, $r^{p,q}_\varpi$ is providing $r^{p,q}_{\varpi,\Nis}$ as the following composition 
\[\xymatrixcolsep{3.4pc}\xymatrix{ H^{p, q}_\Nis(X) \ar[r]^{\nu^{p,q}}\ar@/^1.8pc/[rr]^{r^{p,q}_{\varpi,\Nis}} & H^{p, q}(X) \ar[r]^{r^{p,q}_\varpi}& H^{p,q}_\varpi (X)}\]
where $\nu^{p,q}$ is an isomomorphism with rational coefficients. 

If $X$ is smooth then $H^{p, q}_\Nis(X)\cong H^{p}_\Nis(X, \Z (q)) \longby{\simeq}CH^{q}(X, 2q-p)$ are Bloch's higher Chow groups (see \cite[Thm. 19.1]{VL}). Therefore, for $X$ smooth, we let the cohomological  higher cycle map $$c\ell^{p, q}_\varpi : CH^{q}(X, 2q-p)\to H^{p,q}_\varpi (X)$$ be induced by the composite $r^{p,q}_\varpi\circ \nu^{p,q}= r^{p,q}_{\varpi,\Nis}$ and this latter identification with the higher Chow groups; the previously mentioned cohomological period conjecture, \ie the surjectivity of $r^{p,q}_{\varpi , \Nis}\otimes \Q$ over $K = \bar \Q$, is equivalent to the surjectivity of $c\ell^{p, q}_\varpi$ with rational coefficients.
If $X$ is smooth and projective over $K = \bar \Q$ then the period conjecture for motivic cohomology is equivalent to the vanishing
$$H^{p,q}_\varpi (X)_\Q = H^p(X_\an, \Q_\an (q))\cap H^{p}_\dr(X)=0$$
for $p\neq 2q$ and the surjectivity of the classical cycle class map
$$c\ell^{2q, q}_\varpi : CH^{q}(X)_\Q\to H^{2q}(X_\an, \Q_\an (q))\cap H^{2q}_\dr(X)$$
as explained in \cite[Prop. 1.4.4]{ABVB}. In particular, since both Betti and de Rham cohomologies satisfy hard Lefschetz theorems and Hodge type conditions, the period conjecture implies Grothendieck standard conjectures over $K = \bar \Q$; actually, it implies that numerical equivalence coincide with homological equivalence thus the inverse of the Lefschetz operator is algebraically defined as well as the K\"unneth projectors.

\subsection{Homological periods} In general, for any $X$ algebraic $K$-scheme, we here show that there is also an homological jointly with a Borel-Moore version of the period regulator which fits better with higher cycles.  First note that in addition to $i$-twisted singular $j$-homology  $H_{j, i}^{\an}(X)$ and de Rham homology $H_j^{\dr} (X)$ we have de Rham homology with compact support  $H_j^{\dr,c} (X)$ (see \S \ref{deRhamsect}) along with homological de Rham regulators by Lemma \ref{regderham}; we also have Borel-Moore homology $H_{j, i}^{\an,c}(X)$ (see \S \ref{hompersect}) 
along with a homological regulator by Lemma \ref{compareBM}. Ayoub's version of the Grothendieck period isomorphism induces the following comparison isomorphisms
$$\varpi_{j, i}^X : H_{j, i}^{\an}(X)\otimes_\Z \C\cong H_j^{\dr} (X)\otimes_K \C\ \ \text{\rm and}\ \ \varpi_{j, i}^{X,c} : H_{j, i}^{\an,c}(X)\otimes_\Z \C\cong H_j^{\dr,c} (X)\otimes_K \C$$
in Lemma \ref{APLemma}, which are compatible with the regulator mappings by Proposition \ref{periodsquare}. This provides a corresponding Betti-de Rham and period homology (with values in an appropriate tensor abelian category of homological periods, see \S \ref{BdeR} and \cite[Def. 2.4.1 \& Lemma 2.4.2]{ABVB}) hence 
$$H_{j,i}^\varpi (X):=H_{j, i}^{\an}(X)\cap H_j^\dr (X)\ \ \text{\rm and}\ \ H_{j,i}^{\varpi,c} (X):=H_{j, i}^{\an, c}(X)\cap H_j^{\dr,c} (X)$$ along with regulators
$$r_{j,i}^{\varpi}\colon H_{j,i} (X)\to H_{j,i}^{\varpi}(X)\ \ \text{\rm and}\ \ r_{j,i}^{\varpi,c}\colon H_{j,i}^c (X)\to H_{j,i}^{\varpi, c}(X)$$
from \'etale motivic homology and that with compact support (see \S \ref{hompersect} for details).  We therefore obtain a canonical homological regulator or higher cycle map
$$c\ell_{j,i}^\varpi : CH_{i}(X, j-2i)\to H_{j,i}^{\varpi,c} (X)$$
for any equidimensional scheme $X$ of dimension $d$ and $0\leq i\leq d$ (see Corollary \ref{periodreg}). In fact, now we have the change of topology map $$\nu_{j,i} : H_{j,i}^{BM} (X)\to H_{j,i}^c (X)$$  where $H_{j,i}^{BM} (X)$ is Borel-Moore motivic homology as defined in  \cite[Def. 16.20]{VL} and $$H_{j,i}^{BM} (X)\cong CH_i (X,j-2i)$$  for any such $X$  by \cite[Prop. 19.18]{VL} where $CH_{i}(X, j-2i) := CH^{d-i}(X, j-2i)$. Thus, from $\nu_{j,i}$ we get a higher cycle map $c\ell_{j,i}$ (see \eqref{cycle}) whence $c\ell_{j,i}^\varpi$ is defined as $r_{j,i}^\varpi\circ c\ell_{j,i}$ under the latter identification with higher Chow groups.

The natural isomorphism $H_{j,i}^{BM}(X) \cong CH_i(X, j-2i)$ is established in \cite[Prop. 19.18]{VL} under the assumption that $X$ is a quasi-projective variety. For a general separated scheme $X$ of finite type over $K$, this identification extends unconditionally by standard cdh-descent arguments. Indeed, by Chow's Lemma, there exists a proper birational morphism $\pi: \tilde{X} \to X$ such that $\tilde{X}$ is quasi-projective. Since the cdh-topology forces the excision triangle for the abstract blow-up square associated to $\pi$ and its center, both Borel--Moore motivic homology and Bloch's higher Chow groups satisfy the same long exact localization sequences. By Noetherian induction on the dimension of the singular and non-quasi-projective loci, the isomorphism for $\tilde{X}$ and the lower-dimensional strata canonically induces the claimed natural isomorphism for the general scheme $X$.

\subsection{Comparison between cohomological and homological periods} For any scheme $X$ which is algebraic over $K = \bar \Q$, parallel to the cohomological period conjectures (see \cite[Eq. (1.6)]{ABVB}), it appears natural to expect (Conjecture \ref{periodconj}) the surjectivity of the integrally defined homological period regulators $r_{j,i}^\varpi$ and $r_{j,i}^{\varpi,c}$  as well as for the higher cycle map $c\ell_{j,i}^\varpi$ with rational coefficients. Notably, for $X$ smooth and equidimensional, $d= \dim (X)$, $p= 2d-j$ and $q=d-i$, we have that the Borel-Moore homological and cohomological period regulators correspond under duality and the higher cycle class maps $c\ell^{2d-j, d-i}_\varpi$ and $ c\ell_{j,i}^\varpi$ agree rationally (see Lemma \ref{pdual} and Proposition \ref{regdual}). However, if $X$ is not smooth then the homological and cohomological period regulators do have different sources and targets (see \S \ref{dualsect} for details).

\subsection{Higher zero cycles} For $X$ with arbitrary singularities and higher zero cycles, \ie for $i=0$, we have an isomorphism $H_j^{Sus}(X)\iso H_{j,0} (X)$ with Suslin homology and also $CH_0(X, j)\iso H_{j,0}^c (X)$
over $K = \bar K$ (see \cite[Lemma 13.4.1]{BVK}), in such a way that the homological and cohomological period regulators are  (Corollary \ref{periodreg})
$$r_{j,0}^{\varpi}:H_j^{Sus}(X)\to H_{j,0}^{\varpi} (X),$$
$$ c\ell_{j,0}^\varpi\colon CH_0(X, j) \to H_{j,0}^{\varpi,c} (X)$$
and 
$$r^{2d-j, d}_{\varpi}\colon H^{2d-j, d}(X)\cong H^{2d-j}_\cdh(X, \Z (d)) \to H^{2d-j,d}_\varpi (X)$$
where $H_{j,0}^{\varpi} (X)_\Q\neq H_{j,0}^{\varpi,c} (X)_\Q$ and $H_{j,0}^{\varpi,c} (X)_\Q\neq H^{2d-j,d}_\varpi (X)_\Q$, in general. 

Note that the change of topology maps are isomorphisms strictly when $i = 0$, as established in \cite[Lemma 13.4.1]{BVK}. One might wonder why the comparison properties of effective motivic complexes, such as \cite[Cor. 1.8.5]{BVK}, do not automatically imply similar isomorphisms for general twists $i > 0$. The obstruction lies in the weight filtration and the structure of the Tate object $\mathbb{Z}(i)$. While the derived Albanese functor completely captures the weight $\leq 1$ part of the motive, the higher twisted motivic complexes $\mathbb{Z}(i)$ for $i > 0$ contain deeper, non-trivial higher-weight components that escape the level 1-motivic core. Specifically, the change of topology from the Nisnevich to the \'etale (or $\eh$) setting induces an isomorphism on the level-1 filtration blocks of the motive, but it creates a substantial discrepancy on the higher-weight fragments. Consequently, the perfect match between Nisnevich and \'etale structures is lost for $i > 0$, confining the structural isomorphism exclusively to the case of zero-cycles.

We study these period regulator maps when $j = 0, 1$ using the 1-motivic period conjecture that we discuss below.

\subsection{1-Motivic period conjectures} It is natural to appeal to the existence of the motivic Albanese which provides a suitable context for computations because of the fully faithfulness of the de Rham-Betti realisation of 1-motives as proven in \cite[Theorem 2.7.1]{ABVB}. We can apply the motivic Albanese functor $\LAlb$  to the Voevodsky motives $M(X)$, $M^c(X)$ and $M(X)^*(d)[2d]$, for equidimensional algebraic $K$-schemes $X$ and  $d=\dim (X)$, obtaining the Albanese complex $\LAlb(X)$, the Borel-Moore Albanese complex $\LAlb^c(X)$ and the cohomological Albanese complex $\LAlb^*(X)$, respectively, and their Cartier duals $\RPic(X)$, $\RPic^c(X)$ and $\RPic^*(X)$, see \cite[Def. 8.5.1 \& Def. 8.6.2]{BVK}. Considering $j$-th homologies $\LA{j}(X)$, $\LA{j}^c(X)$ and $\LA{j}^*(X)$ and applying the period realisation of 1-motives $T_\omega$ (from \cite[\S 1.3 \& Cor. 3.1.3]{ABVB}), with rational coefficients, we thus obtain (see \S \ref{BdeR} for details)
\[\xymatrixcolsep{4pc}\xymatrix{ 
H_j^{Sus}(X)_\Q\ar[d]^{a\ell b_j}\ar[r]^{r^\varpi_{j,0}}&  H_{j,0}^{\varpi} (X)_\Q\ar[d]^{}\\ 
H_{j,0}^{(1)}(X)_\Q\ar@{.>}[r]^{r_{j,0}^{\varpi, (1)}}&  T_\omega^\Q(\LA{j}(X))}\]
and 
\[\xymatrixcolsep{4pc}\xymatrix{ 
CH_0(X, j)_\Q\ar[d]^{a\ell b_j^c}\ar[r]^{c\ell_{j,0}^\varpi}&  H_{j,0}^{\varpi,c} (X)_\Q\ar[d]^{}\\ 
H_{j,0}^{(1),c}(X)_\Q\ar@{.>}[r]^{c\ell_{j,0}^{\varpi, (1)}}&  T_\omega^\Q (\LA{j}^c(X))
}\]
where the maps $a\ell b_j$ and $a\ell b_j^c$ to $1$-motivic homologies (constructed in \cite[\S 13.1 ]{BVK}) are induced by the motivic Albanese map (see \cite[(13.1.1)]{BVK}, \eqref{hcyclemap} and Lemma \ref{cohreproj} for details).
Dually, we get 
\[\xymatrixcolsep{5pc}\xymatrix{ 
H^{2d-j}_\cdh(X, \Q (d))\ar[d]^{a\ell b_j^*}\ar[r]^{r^{2d-j, d}_{\varpi}}&  H^{2d-j,d}_\varpi (X)_\Q\ar[d]^{}\\ 
H^{2d-j}_{\cdh, (1)}(X, \Q (d))\ar@{.>}[r]^{r^{2d-j, d}_{\varpi , (1)}}&  T_\omega^\Q (\LA{j}^*(X))
}\]
where $a\ell b_j^*$ (as in \eqref{cohalbmap})  is now the cohomological motivic Albanese map (see \cite[\S 13.7]{BVK}). In general, we can show that the 1-motivic period conjecture holds true rationally: the maps $r_{j,0}^{\varpi, (1)}$, $c\ell_{j,0}^{\varpi, (1)}$ and $r^{2d-j, d}_{\varpi , (1)}$ are surjections for all $j\geq 0$ over $K=\bar \Q$ (see Theorem \ref{1-motivic}).

\subsection{Case $j=0,1$} For $j = 0, 1$ we know that (modulo torsion) $H_{j,0}^\varpi (X)\cong T_\omega (\LA{j}(X))$, $H_{j,0}^{\varpi,c} (X)\cong T_\omega (\LA{j}^c(X))$ and $H^{2d-j,d}_\varpi (X)\cong  T_\omega (\LA{j}^*(X))$ are 1-motivic for $d=\dim (X)$. To make use of the previous factorisation of the homological and cohomological period regulators via $a\ell b_j$, $a\ell b_j^c$ and $a\ell b_j^*$ we further need to know that such maps are surjections. This is the case when $M(X)$, $M^c(X)$ or $M(X)^*(d)[2d]$ are 1-motivic,  \eg for curves, so that the named maps are isomorphisms but it is not clear in general.

The period conjectures are verified for $j=0$ (Theorem \ref{zero}) but for $j=1$ we just can characterise the cokernels of $r^{2d-1, d}_{\varpi}\otimes\Q/\Z$ and $r_{1, 0}^{\varpi}\otimes\Q/\Z$ and $c\ell_{1,0}^\varpi\otimes\Q/\Z$ in terms of torsion classes in $H^{2d}_\cdh(X, \Z (d))$, $H_0^{Sus}(X)$ or $CH_0(X)$, respectively (Theorem \ref{one}). It means that Grothendieck period conjectures imply very general Ro\u{\i}tman theorems over $K=\bar \Q$ at least.

However, unconditionally, for $j=1$, $\LA{1}(X)=[L_1\by{u_1} G_1]$ is the Albanese 1-motive \cite[Cor. 12.6.6]{BVK},  $\LA{1}^c(X)=[L_1^c\by{u_1^c} G_1^c]$ is the Borel-Moore Albanese \cite[\S 8.5 \& \S 12.11]{BVK} and $\LA{1}^*(X)=[L_1^*\by{u_1^*} G_1^*]\cong \Alb^+(X)$, which is the cohomological Albanese in \cite[\S 3.1]{BVS} (the isomorphism follows from \cite[Thm. 12.12.6]{BVK}): the geometry of these 1-motives is very well understood and it fully describe the free part of $H_{1,0}^\varpi (X)$, $H_{1,0}^{\varpi,c} (X)$ and $H^{2d-1,d}_\varpi (X)$ as $\ker u_1$, $\ker u^c_1$ and $\ker u^*_1$ over $K=\bar \Q$, respectively (see Lemma \ref{11}). In particular, it follows that $H^{2d-1}(X_\an, \Z_\an (d))_\fr\cap H^{2d-1}_\dr(X)=0$ if $X$ is proper and $H_1(X_\an, \Z)_\fr\cap H_1^\dr(X)=0$ if $X$ is  normal. 

Finally, the previous computations provide the following general picture which makes sense for any $X$ over an algebraically closed field $K=\bar K$. (The assumption that $K$ is a subfield of the complex numbers is not needed in the following). Let $H_0^{Sus}(X)_0$, $CH_0(X)_0$ and $H^{2d}_\eh(X, \Z(d))_0$ be the kernels of the surjections as in Lemma \ref{cohreproj} (see also in Theorem \ref{zero} and Remark \ref{zerork}). 
Let $\LA{1}(X)(K) = \coker u_1$, $\LA{1}^c(X)(K) = \coker u_1^c$ and  $\LA{1}^*(X)(K) = \coker u_1^*$ be the groups given by taking the $K$-points of the 1-motives $\LA{1}(X)$, $\LA{1}^c(X)$ and $\LA{1}^*(X)$, respectively. For any algebraic $K$-scheme $X$ we have canonical generalised Albanese mappings (see  Proposition \ref{highalb}): the plain Albanese on Suslin homology
$$a\ell b_0:H_0^{Sus}(X)_0\to \LA{1}(X)(K),$$
the Borel-Moore Albanese map
$$a\ell b_0^c:CH_0(X)_0\to \LA{1}^c(X)(K)$$
and the cohomological Albanese map
$$a\ell b_0^*:H^{2d}_\cdh(X, \Z(d))_0\to \LA{1}^*(X)(K)$$
which are always surjections on torsion subgroups (by \cite[Chap. 13 ]{BVK} as in Theorem \ref{one}). We may actually conjecture (see the Core Equivalence Conjecture \ref{coreq}) that these Albanese mappings $a\ell b_{0}$, $a\ell b^c_{0}$ and $a\ell b_{0}^*$ would be surjections with a uniquely divisible kernel providing suitable definitive versions of Ro\u{\i}tman Theorems (to be also compared with \cite{Kho} and \cite{Gei}).

Recall that we have a surjection from Suslin homology $H_0^{Sus}(X)$ to $CH_0(X)$ (\cite[Ex. 17.3(b)]{VL}) and that $H_0^{Sus}(X)_\tors \cong\LA{1}(X)(K)_\tors$ if $X$ is normal by \cite[Cor. 13.5.3]{BVK} and \cite[Thm. 1.1]{Gei} where $\LA{1}(X)\cong \Alb^-(X)$ is the Serre-Albanese (by \cite[Prop. 12.7.2]{BVK}), the homological Albanese semiabelian variety in \cite[\S 5.1]{BVS}. In particular,  $a\ell b^c_0$ is an isomoprphism on torsion for $X$ normal and proper. Moreover, there is a morphism from Levine-Weibel cohomological Chow group $CH^d_{LW}(X)$ to $H^{2d}_\cdh(X,\Z(d))$, which is a surjection with divisible kernel at least when $X$ is projective, and $CH^d_{LW}(X)_\tors \cong \LA{1}^*(X)(K)_\tors$, see \cite[Thm. 1.7]{BKri}. Actually, in this case we have that $\LA{1}^*(X)(K)\cong \Alb^+(X)$ is the group of $K$-points of a semiabelian variety and a map from $CH^d_{LW}(X)_0$ to $\Alb^+(X)$
has been previously constructed in \cite[Thm. 6.4.1]{BVK} (showing that it is the universal map to semi-abelian varieties). 

\section{Homological regulators and periods}
\subsection{Motivic complexes}
Fix a base field $K$ of zero characteristic. We are mostly interested to work over $K = \bar \Q$ the field of algebraic numbers but going to just assume that $K=\bar K$ is algebraically closed if suitable.

We refer to \cite{VL}, \cite{AICM}, \cite{BVK} and \cite{BVA} for constructions and properties of the triangulated categories of motivic complexes over the field $K$.   

Let $\DM_\tau^\eff$ be the triangulated category of Voevodsky effective motivic complexes of $\tau$-sheaves over $K$ for $\tau = \Nis$ or $\et$, the Nisnevich or \'etale topology, respectively (\eg see \cite{AICM}, \cite[Lect. 9 \& 14]{VL} and \cf \cite[\S 2.2.6]{BVA}). We here shall denote $\Z (q)$ the usual motivic complex for Nisnevich, \'etale topology or its $\eh$ sheafification after forgetting transfers, as explained in \cite[\S 10.2]{BVK}. Similarly, for any $X$ algebraic $K$-scheme, we set $M (X)$ for the Nisnevich or the \'etale effective motive (with or without transfers) and $M^c (X)$ for the Nisnevich or \'etale effective motive with compact support (see \cite[\S 2.2.6]{BVA} and \cf \cite[Def. 16.13]{VL} and \cite[Def. 1.8.6 \& \S 8.1]{BVK}). We have a canonical map $M (X)\to M^c (X)$ which is an isomorphism when $X$ is proper. Actually, all these motives with respect to the Nisnevich or the \'etale topology correspond each other under 
\begin{equation}\label{alpha}
    \alpha : \DM_\Nis^\eff\to \DM_\et^\eff
\end{equation}
the change of topology triangulated monoidal functor and the change of topology $\alpha$ is an equivalence with rational coefficients (see \cite[Remark 14.3 \& Thm. 14.30]{VL} where the proof for bounded below applies here as well) after noting that $\DM_\tau^\eff$ is generated by $M (X)$ for $X$ smooth. Recall that considering the triangulated category of (effective) motivic complexes without transfers $\DA^\eff_\et$  we have a canonical functor 
$$\upsilon : \DM^\eff_\et\to \DA^\eff_\et$$
induced by forgetting transfers, \ie by restriction of additive presheaves on finite correspondences to smooth schemes (see \cite{AICM} and \cite[\S 2.3.2]{AHop}); moreover, by adding transfers we get an adjoint of $\upsilon$ inducing an equivalence with rational coefficients and together with $\alpha$ we obtain equivalences
$$\DM_{\Nis ,\Q}^\eff\cong \DM_{\et,\Q}^\eff\cong \DA_{\et,\Q}^\eff$$ 
of tensor triangulated categories (see \cite[Cor. B.14]{AHop} and \cite[Thm. 14.30]{VL}). 

By inverting the Tate object $\mathbb{Z}(1) \df \mathbb{G}_m[-1]$, which resides natively in $\DM^\eff_\tau$, we obtain the stable category $\DM_\tau$ (see \cite[Lect. 14]{VL}). For any $M\in \DM^\eff_\tau$ and $q \in \mathbb{Z}$, we denote $M(q)\df M\otimes \Z(q)$. The transition to the stable framework is governed by the infinite suspension functor $\Sigma^\infty \colon \DM^\eff_\tau \to \DM_\tau$. By Voevodsky's cancellation theorem \cite[Thm. 16.25]{VL}, the twisting functor $M \mapsto M(1)$ is fully faithful on the effective category, which crucially implies that $\Sigma^\infty$ itself is fully faithful. Therefore, the structural Hom-sets computed in $\DM^\eff_\tau$ are naturally identified with those in the stable category $\DM_\tau$, ensuring that the stabilization preserves the underlying level 1-motivic extensions.

Similarly, we obtain $\DA_{\et}$ (see \cite{AICM}). Actually, $\DM_{\tau}$ (and $\DA_{\et}$)  can be better obtained directly as the Verdier localization of the derived category of $L$-spectra of $\tau$-sheaves with (and without) transfers over smooth $K$-schemes, where $L$ is a Lefschetz motive, as Ayoub explained in \cite[\S 2.3 \& 4.1]{AICM} and \cite[\S 2.1.1]{AHop}. Following Ayoub, there are triangulated and monoidal (infinite suspension) functors
$$\DA_{\et}^\eff\to \DA_{\et}\hspace*{1cm} \DM_{\tau}^\eff\to \DM_{\tau}$$
both provided with a right adjoint (see \cite[\S 2.1 \& 2.3]{AHop}). The previous equivalences extend to the stable version
$$\DM_{\Nis ,\Q}\cong \DM_{\et,\Q}\cong \DA_{\et,\Q}$$ 
and we shall make reference to all these categories as categories of motivic complexes. For the sake of notation we shall indicate the motivic complexes $M (X)$ and $M^c (X)$ regarded as objects of $\DM_{\tau}$ or $\DA_\et$ as well.\footnote{$M (X):=\Sigma^\infty_{L_\tr}\Z_\tr (X)\in \DM_{\et}$ with notation in \cite{AICM}}

\subsection{Motivic homology}
For any $M\in\DM^\eff_{\Nis}$ we then have canonical maps induced by the change of topology functor
$$H^{p,q}_{\Nis} (M)\df \Hom_{\DM^\eff_{\Nis}} (M , \Z (q)[p]) \longby{\nu^{p,q}} H^{p,q} (M)\df \Hom_{\DM^\eff_{\et}} (\alpha M , \Z (q)[p]).$$
In particular, if $X$ is an algebraic $K$-scheme, for $M= M(X)$ we get canonical maps from motivic cohomology  $H^{p,q}_{\Nis} (X)$ to \'etale motivic cohomology $H^{p,q} (X)$. Further, we have $H^{p,q} (X)\cong  H^{p}_{\eh} (X, \Z (q))$ and $H^{p}_{\eh} (X, \Z (q))\cong H^{p}_{\et} (X, \Z (q))$ if $X$ is smooth (see \cite[Prop. 10.2.3]{BVK}). 
Similarly, for any $M\in \DM_{\Nis}$  we have
$$H_{j,i}^{\Nis} (M)\df \Hom_{\DM_{\Nis}} (\Z (i)[j], M) \longby{\nu_{j,i}} H_{j,i} (M)\df \Hom_{\DM_{\et}} (\Z (i)[j],\alpha M)$$
and for $M = M^c(X)$ we get the following mapping
$$H_{j,i}^{BM} (X)\df \Hom_{\DM^\eff_{\Nis}} (\Z (i)[j],M^c (X)) \longby{\nu_{j,i}} H_{j,i}^c (X)\df \Hom_{\DM^\eff_{\et}} (\Z (i)[j], M^c (X))$$
from the usual Borel-Moore motivic homology (see \cite[Def. 16.20]{VL}) to the \'etale one. For $M = M(X)$ we also get the mapping
$$H_{j}^{Sus} (X)\df \Hom_{\DM^\eff_{\Nis}} (\Z [j],M (X)) \longby{\nu_{j}} H_{j} (X)\df \Hom_{\DM^\eff_{\et}} (\Z [j], M (X))$$ 
from Suslin homology to \'etale motivic homology. Recall that for any equidimensional scheme $X$ of dimension $d$ and $0\leq i\leq d$ we have a natural isomorphism $$H_{j,i}^{BM} (X)\cong CH_i (X,j-2i)$$ where $CH_{i}(X, j-2i) = CH^{d-i}(X, j-2i)$ is Bloch's higher Chow group  (see \cite[Prop. 19.18]{VL}). We therefore have a canonical map
\begin{equation}\label{cycle}
    c\ell_{j, i} : CH_{i}(X, j-2i)\to H_{j,i}^c (X)
\end{equation}
induced by the change of topology. Recall that for a closed subscheme $Y$ of a scheme $X$ we have a canonical triangle $M^c(Y)\to M^c(X) \to M^c(X-Y)\to M^c(Y)[1]$
yielding a long exact sequence 
$$H_{j,i}^c (Y)\to H_{j,i}^c (X)\to H_{j,i}^c (X-Y)\to H_{j-1,i}^c (Y)$$
which is compatible with $c\ell_{j, i}$. Note that $c\ell_{j, i}\otimes\Q$ is an isomorphism for any $X$ algebraic $K$-scheme; we have $H_{j,i}^{BM} (X)_\Q\cong H_{j,i}^c (X)_\Q$ as well as $H^{p,q}_{\Nis} (X)_\Q\cong H^{p,q}(X)_\Q$ with rational coefficients. For $X$ proper $H_{j,0}^{BM}(X)= H_{j}^{Sus}(X)$ and $H_{j,0}^c (X)= H_{j} (X)$.

\subsection{de Rham homological regulator}\label{deRhamsect}
Consider $\uOm \in \DM^\eff_\tau$ (resp.\/ $\uOm \in \DA^\eff_\et$) which is given by the complex of presheaves on finite correspondences (resp.\/ on the category smooth $K$-schemes) that associates to $X$ smooth over $K$ the global section $\Gamma (X, \Omega^\d_{X/K})$ of the usual algebraic de Rham complex (see \cite[\S 1.1]{ABVB}, \cite[\S 2.1]{LW} and \cite[\S 2.3]{AHop}). Recall that we have a canonical de Rham regulator $r^q : \Z (q) \to  \uOm$ for all $q\geq 0$.  This is given by the augmentation $r^0\df \epsilon : \Z (0) =\Z \to \uOm$ for $q=0$ and the  dilogarithm $r^1 \df d\log: \Z (1) = \G_m[-1]  \to \uOm$ for $q=1$. Now  $r^q$  for $q\geq 2$ is obtained from the fact that $\uOm$ is an algebra (\eg see \cite[Thm. 5.1]{LWdr}), \ie  we have a map $\uOm \otimes \uOm \to \uOm$ given by the wedge product of differential forms.  As $\Z(q)\cong \Z (1)^{\otimes q}$ for $q\geq 2$ we can define  $r^q$ as the composition of $d\log^{\otimes q} : \Z (1)^{\otimes q} \to \uOm^{\otimes q}$ and $\uOm^{\otimes q}\to \uOm$. 
Recall  that $\uOm \in \DM^\eff_\et$ represents de Rham cohomology. Setting
$$H^{p}_{\dr}(M) \df \Hom_{\DM^\eff_\et} (M, \uOm [p])$$ 
for $M\in \DM^\eff_\et$, for any algebraic scheme $X$ and $M=M(X)\in \DM_\et^\eff$ we have  (see \cite{ABVB})
$$H^p_\dr(X)\df \Hom_{\DM_\et^\eff} (M(X), \uOm [p]) \cong H^p_\eh(X, \uOm).$$

Recall that we have an isomorphism $\ihom_{\et} (\Z (q), \uOm) \cong \uOm$ in $\DM^\eff_{\et}$ for all $q\geq 0$. This identification follows from a direct computation combined with $\mathbb{A}^1$-homotopy invariance. Indeed, for $q=1$, since $\mathbb{Z}(1) \cong \mathbb{G}_m[-1]$, the monoidal adjunction identifies $\ihom_{\et}(\Z(1), \uOm)$ with the complex of internal sheaves representing the de Rham cohomology of the torus $\mathbb{G}_m$. In characteristic zero, a direct computation of the algebraic de Rham complex yields $\underline{\Omega}(U \times \mathbb{G}_m) \cong \underline{\Omega}(U) \oplus \underline{\Omega}(U)[-1], $ where the shifted summand is generated by the invariant differential form $d\log(t)$. By $\mathbb{A}^1$-homotopy invariance, the internal Hom functor against the twist extracts precisely this log-component while canceling the topological shift. This immediately provides the canonical isomorphism $\ihom_{\et}(\mathbb{Z}(1), \underline{\Omega}) \cong \underline{\Omega}$, and the general case follows by induction on $q \geq 0$. This implies that $H^{p}_{\dr}(M(q))\cong H^{p}_{\dr}(M)$ for any non negative twist $q\geq 0$. 

Moreover, following \cite[\S 2.3.1]{AHop}, consider the stable version $\mathbf{\Omega} \df \{\underline{\Omega}_n\}_{n\in \mathbb{N}}$ of the algebraic de Rham complex, which forms a symmetric $T$-spectrum (where $T \df \mathbb{Z}(1)[1] \cong \mathbb{G}_m$). For each $n \geq 0$, the $n$-th level of the spectrum is explicitly given by $\underline{\Omega}_n \df \underline{\Omega}[n]$. The structure maps $\sigma_n \colon \underline{\Omega}_n \otimes T \to \underline{\Omega}_{n+1}$ are induced by the canonical identification 
\[ \underline{\Omega}[n] \otimes \mathbb{G}_m \cong \underline{\Omega}[n+1] \oplus \underline{\Omega}[n] \]
combined with the internal projection onto the $d\log$-component and the monoidal multiplication on $\underline{\Omega}$. Under this setting, the stable de Rham complex defines a highly structured symmetric ring spectrum. The twist cancellation induces a canonical isomorphism $\mathbf{\Omega}(q) \longby{\sim} \mathbf{\Omega}$ in the stable motivic category $\DM_\tau$ (resp.\/ in $\DA_\et$) for any integer $q \in \mathbb{Z}$, directly extending the effective collapsing rules to negative twists. 
To detail this mechanism, we observe that the monoidal inversion of $\mathbb{Z}(1)$ forces a canonical identification between the stable twisted components. For positive twists ($q \geq 0$), the statement follows directly from the effective internal Hom isomorphism established above. For negative twists ($q < 0$), the stable twisted spectrum $\mathbf{\Omega}(q)$ is natively represented by shifting the levels of the symmetric spectrum via the bonding operations. Since the structural maps $\sigma_n \colon \underline{\Omega}_n \otimes \mathbb{Z}(1) \longby{\sim} \underline{\Omega}_{n+1}$ are strict isomorphisms in the stable category, the cancellation of a negative twist is precisely achieved by the stable shift of the levels of the spectrum, which cancels the topological offset by means of the structural log-components. This ensures that the canonical collapse extends uniquely to all negative twists.

Crucially, this remains perfectly valid under the Nisnevich topology ($\tau = \text{Nis}$), i.e., inside the category $\DM_{\text{Nis}}$. Indeed, the algebraic de Rham complex $\underline{\Omega}$ is natively an $\mathbb{A}^1$-invariant Nisnevich sheaf with transfers, and the algebraic decomposition of the torus via $d\log(t)$ forms is local, ensuring that the structural projections satisfy Nisnevich excision: the stable spectrum $\mathbf{\Omega}$ exhibits the exact same canonical stability properties under both topologies. For $M\in \DM_\et$ we set 
$$H_{j}^{\dr}(M) \df \Hom_{\DM_\et} (\Z[j], M\otimes \bOm),$$ for all $j\geq 0$. For $M= M^c(X)$ we here set
$$H_{j}^{\dr, c}(X) \df \Hom_{\DM_\et} (\Z[j], M^c(X)\otimes \bOm),$$ 
for de Rham homology with compact support. Note that we have
$$\Hom_{\DM_\et} (M(X), \bOm [p])\cong \Hom_{\DM_\et^\eff} (M(X), \uOm [p])=H^p_\dr(X)$$
by the adjunction between effective and non effective (see \cite{AICM} and \cite[\S 2.3.2]{AHop}).\footnote{As explained in \cite[Rmk. 2.9]{AICM}, ${\rm Ev}_0 : \cE \mapsto \cE_0$ is right adjoint to ${\rm Sus}^0_{L_\tr}:=\Sigma^\infty_{L_\tr}$ in such a way that ${\rm Ev}_0\bOm = \uOm$ and by our convention $M (X): =\Sigma^\infty_{L_\tr}\Z_\tr (X) \in \DM_\et$}
Moreover, we get an isomorphism
$$ \tau^\dr_{j, i}: \Hom_{\DM_{\et}} (\Z (i)[j], M\otimes \bOm )  \longby{\simeq} \Hom_{\DM_{\et}}  (\Z [j], M\otimes \bOm ),$$ 
that is $H_{j}^{\dr}(M (i)) \cong H_{j}^{\dr}(M)$, for all $i$ and $j$. Note that the augmentation $\epsilon : \Z \to \bOm$ yields a canonical map 
$$\epsilon_M : M\to M\otimes \bOm$$
obtained as $M\otimes  \epsilon$, yielding the following homological {\it de Rham regulator} mapping.

\begin{lemma} \label{regderham}
For any $M\in \DM_\et$ we have a 
$r_{j, i}^\dr : H_{j, i}(M)\to H_{j}^{\dr}(M)$ for all $i, j\geq 0$. In particular, we get $H_{j,i}^{BM} (X)\to H_{j}^{\dr, c}(X) $ for any scheme $X$ over $K$.  
\end{lemma}
\begin{proof}  
Given a map $\Z (i)[j]\to M$ its composition with $\epsilon_M$ yields a map $\Z (i)[j]\to M\otimes \bOm$ and we then make use of $ \tau^\dr_{j, i}$ above. For $M = M^c(X)$, composing with $\nu_{j, i}$ we further get the regulator from Borel-Moore motivic homology. 
\end{proof}
For $X$ equidimensional $d = \dim (X)$ and $0\leq i\leq d$ we get
$$c\ell_{j, i}^\dr : CH_{i}(X, j-2i)\to H_{j}^{\dr, c} (X).$$ 
\subsection{Homological period regulator}\label{hompersect}
Consider an embedding $\sigma: K \into \C$ and the corresponding stable version of the Betti realization (see \cite[Def. 2.4]{AHop} and \cite[Def. 2.1]{A})
$$
\beta_\sigma : \DM_\et\to D(\Z)
$$
which is a triangulated functor in the derived category of abelian groups $D(\Z)$ such that $\beta_\sigma (\Z) =\Z$ and  $\beta_\sigma (\Z (q)) \cong \Z_\an(q)\df (2\pi i)^q\Z$, $q\geq 1$. Recall that the functor $\beta_\sigma$ admits a right adjoint $\beta^\sigma : D(\Z) \to \DM_\et$ (see \cite[Def 2.4]{AHop} and \cite[Def. 1.7]{AHop2}).\footnote{The functors $\beta_\sigma$ and $\beta^\sigma$ are respectively denoted ${\rm Btti}^*$ and ${\rm Btti}_*$ on $\DA_\et$ in \cite{AHop}} For $M\in \DM_\et$ we set
$$H^{p,q}_\an (M) \df \Hom_{\DM_\et}(M, \beta^\sigma (\Z_\an (q))[p])\cong  \Hom_{D (\Z) }(\beta_\sigma (M),  \Z_\an (q)[p]).$$
Moreover, for $M\in \DM_\et$,  we denote
$$H_{j,i}^\an (M) \df \Hom_{\DM_\et}(\Z (i)[j], M\otimes \beta^\sigma (\Z)).$$
Denote $r_\sigma^M: M \to \beta^\sigma\beta_\sigma (M)$ the unit of the adjunction. Denote 
$$\rho_M : M \to M\otimes \beta^\sigma(\Z)$$
the map $M\otimes r_\sigma^\Z$ induced by the unit. Note that the universal properties of the stable Betti adjunction $(\beta_\sigma, \beta^\sigma)$ ensure the identity $\beta^\sigma\beta_\sigma (M) \cong  M \otimes \beta^\sigma(\Z)$, thus identifying $\rho_M$ with $r_\sigma^M$, the unit of the adjunction. This canonical identification 
 is an immediate consequence of the projection formula for the stable Betti realization: since $\beta_\sigma$ is a strong monoidal triangulated functor preserving compact objects, its right adjoint \(\beta ^{\sigma }\) satisfies the global projection formula \(M \otimes \beta^\sigma(C) \cong \beta^\sigma(\beta_\sigma(M) \otimes C)\) unconditionally for any motive $M\in \DM_\et$ (\cf \cite{A}, \cite[Prop. 2.7 \& Lemma 2.8]{AHop}); setting \(C = \mathbb{Z}\) yields the claimed identity.

This map $\rho_M $ yields a homological  {\it Betti regulator} mapping
 $$
r_{j,i}^{\an}:H_{j,i} (M)\to H_{j,i}^\an (M).
 $$
which is just given by composition. In particular, for $M= M^c(X)$ we obtain a mapping from motivic \'etale homology as shown  
$$H_{j,i}^c (X)\to H^{\an, c}_{j, i} (X)\df \Hom_{\DM_\et}(\Z (i)[j], M^c(X)\otimes \beta^\sigma (\Z))$$ and via $\nu_{j, i}$ a mapping from Borel-Moore motivic homology. For any equidimensional scheme $X$ of dimension $d$, $0\leq i\leq d$, we then get a cycle map
$$ c\ell_{j,i}^\an:CH_i (X,j-2i) \to H^{\an,c}_{j, i} (X).$$
In fact, for such an $M= M^c(X)$ we have $H_{j,i}^{BM} (X)\cong CH_i (X,j-2i)$ and the composition of the regulator 
$r_{j,i}^{\an}$ with the change of topology map $\nu_{j, i}$ displayed above induces the claimed cycle map.\begin{lemma} \label{compareBM}
For a field homomorphism $\sigma : K \into \C$ and any scheme $X$ over $K$ we let $X_\an$ be the analytic space of $\C$-points. There are canonical isomorphisms
$$H^{\an}_{j, i} (X)_\Q\cong H_j(X_\an, \Q_\an)\ \ \text{and}\ \ H^{\an,c}_{j, i} (X)_\Q\cong H_j^{BM}(X_\an, \Q_\an)$$
with the usual singular and  Borel-Moore homology of the analytic space.
\end{lemma}
\begin{proof}
Let $\bar X$ be a proper scheme such that $X \df \bar X - Y$ for $Y \subset \bar X$ a closed subscheme. We have that $M(\bar X) = M^c(\bar X)$ and $M (Y) =M^c(Y)$ and $M^c (X)$ is the cone of $M (Y)\to M (X)$. We also have that $H_j^{BM}(X_\an, \Z_\an(i))\cong H_j(\bar X_\an, Y_\an; \Z_\an(i))$. There are induced maps $M (+ )\otimes \beta^\sigma (\Z)\to  M^c ( + )\otimes \beta^\sigma (\Z)$ and by the canonical triangle of the pair we are then left to show that the complex $M( + )\otimes \beta^\sigma (\Q)$ in $\DM_{\et, \Q}\cong \DA_{\et, \Q}$ is representing the singular chain complex of $ + _\an$, which in fact follows directly from Ayoub's construction. Indeed, $\beta_\sigma M(+)$ computes exactly the singular chain complex of the underlying analytic space by definition of the stable Betti realization functor (see \cite{A} and \cf \cite[\S 4.2]{AICM}). Since $\beta^\sigma\beta_\sigma (M ( + ))\cong  M ( + ) \otimes \beta^\sigma (\Q)$ we then obtain
$$ \Hom_{\DM_{\et, \Q}}(\Z (i)[j], M (+ )\otimes \beta^\sigma (\Q))\cong \Hom_{D (\Z )}(\Z_\an (i)[j], \beta_\sigma M (+))_\Q\cong H_j (X_\an, \Q_\an)$$
and the claimed isomorphism follows.
\end{proof}
Actually, we get: 
\begin{lemma}[\protect{Ayoub}] \label{APLemma} For $M\in \DM_\et$
there is a canonical isomorphism 
$$\varpi^{M}: M\otimes \beta^{\sigma} (\Z)\otimes_\Z \C \longby{\sim}M\otimes \bOm \otimes_K \C$$
whose composition with $\rho_M$ is induced by the augmentation $\epsilon_M : M\to M\otimes\bOm $ after tensoring with $\C$, \ie the following diagram
\[\xymatrix{M\ar[d]_{\epsilon_M} \ar[r]_{\rho_M}& M\otimes \beta^{\sigma} (\Z)\ar[r]&M\otimes \beta^{\sigma} (\Z)\otimes_\Z \C\ar[dl]^{\varpi^{M}}\\
M\otimes \bOm \ar[r]^{}&M\otimes \bOm \otimes_K \C&
 }\]
 commutes. 
\end{lemma}
\begin{proof} This $\varpi^{M}$ is obtained by tensoring with $M$ the classical Grothendieck period comparison isomorphism $\varpi^{\Z}:\beta^\sigma\beta_\sigma (\Z)\otimes_\Z \C \longby{\sim}\bOm \otimes_K \C$ (see \cite[Thm. $1^\prime$]{GrdR}, \cite[Cor. 2.89 \& Prop. 2.92]{AHop} and also \cite[\S 3.5]{APKZ}). To establish the commutativity of the diagram and detail the underlying construction, we also observe that the entire framework for an arbitrary geometric motive $M$ is obtained by tensoring with $M$ the universal period $r_\sigma^\Z$ defined on the unit object $\mathbb{Z}$. 
Applied to the unit object $M= \Z$ we obtain the following square
\[\xymatrix{\Z\ar[d]_{\epsilon_\Z} \ar[r]^{r_\sigma^\Z\hspace{1cm}}& \beta^{\sigma} \beta_\sigma(\Z)\otimes_\Z \C\ar[d]^{\varpi^{\Z}}\\
\bOm \ar[r]^{}&  \bOm \otimes_K \C&
 }\]
 whose commutativity is clear. By tensoring with $M$ we therefore obtain the claimed commutative square for any $M\in \DM_\et$. 
\end{proof}
For $M\in \DM_\et$, by composition with $\varpi^M$, using $\tau^\dr_{j, i}$, we get a period isomorphism 
$$\varpi_{j,i}^M: H_{j,i}^\an (M)\otimes_\Z \C \longby{\simeq} H^\dr_j(M)\otimes_K\C$$
and a period regulator as follows.
\begin{propose} \label{periodsquare}
For  $M\in \DM^\eff_\et$ along with a fixed embedding $\sigma: K \into \C$ the period isomorphism $\varpi_{j,i}^M$ above induces a commutative diagram
 \[\xymatrix{ H_{j,i}(M) \ar[rr]^{r^\an_{j,i}} \ar[d]_{r_{j,i}^\dr}&& H_{j,i}^\an (M)\ar[d]^{\iota_{j,i}^\an}\\
H^\dr_j(M)\ar[rr]^{\iota_{j,i}^\dr\hspace*{2cm}}&& H^\dr_j(M)\otimes_K \C \cong H_{j,i}^\an (M)\otimes_\Z \C \hspace*{2cm}
 }\]
\end{propose}
\begin{proof} This easily follows from Lemma \ref{APLemma}. In fact, by construction, the claimed commutative diagram can be obtained from the following commutative diagram:
\[\xymatrix{\Hom_\DM (\Z(i)[j], M) \ar[d]_{\epsilon_M\circ -}\ar[rr]^-{\iota_{i,j}^\an\circ\, \rho_M} & & \ar[d]_{\varpi^M\circ -}\Hom_{\DM}( \Z(i)[j],M\otimes \beta^\sigma (\Z))_\C\\
\Hom_\DM (\Z (i) [j], M\otimes \bOm )\ar[d]_{\tau^\dr_{j, i}}\ar[rr]_-{} & & \Hom_{\DM}(\Z (i) [j], M\otimes \bOm )_\C\ar[d]_{\tau^\dr_{j, i}}\\ 
\Hom_\DM (\Z [j], M\otimes \bOm )\ar[rr]_-{\iota_{j,i}^\dr} & & \Hom_{\DM}(\Z [j], M\otimes \bOm )_\C} \]
\end{proof}
We set $H_{j,i}^\varpi(M):= H_{j,i}^\an (M)\cap H^\dr_j(M)$ for the inverse image of $H^\dr_j(M)$ inside $H_{j,i}^\an (M)$ under $\iota_{j,i}^\an$ and the period isomorphism. In particular, for $M = M(X)$ or $M^c(X)$ we set 
$$H_{j,i}^\varpi(X) := H_{j,i}^\an (X)\cap H^\dr_j(X)\ \ \text{and}\ \ H_{j,i}^{\varpi,c}(X) := H_{j,i}^{\an,c} (X)\cap H^{\dr,c}_j(X)$$ and  from Proposition \ref{periodsquare} we have: 
\begin{cor}\label{periodreg} For $M\in \DM^\eff_\et$ we  then have $ r_{j, i}^\varpi : H_{j,i}(M)\to H_{j,i}^\varpi(M)$. For any $X$ we have $ r_{j, i}^\varpi : H_{j,i}(X)\to H_{j,i}^\varpi(X)$
and if $X$ is equidimensional of dimension $d$, for $0\leq i\leq d$ there is a map
$$ c\ell_{j,i}^\varpi :CH_i (X,j-2i) \to H^{\varpi, c}_{j, i} (X).$$
\end{cor}
\begin{conj}\label{periodconj}
Over $K=\bar \Q$ the period regulator $r_{j, i}^\varpi : H_{j,i}(M)\to H_{j,i}^\varpi(M)$ is surjective for every $M\in \DM^\eff_\et$. 
\end{conj}
\subsection{Duality}\label{dualsect}
Recall that the subcategory $\DM_\gm \subset \DM_\Nis^-\subset \DM_\Nis$ of constructible or geometric motives is rigid (see \cite[Thm 20.17]{VL}): it is provided with an internal $\ihom$ and a dual $M^* = \ihom (M, \Z)$ (see \cite[Def. 20.15]{VL}) such that for any $X$ smooth and $\dim (X)\leq d $ we have 
$$M(X)^* =\ihom_\Nis (M (X), \Z (d))(-d)$$
where $\ihom_\Nis$ is the effective partial internal Hom in $\DM_\Nis^\eff$. Thus, by Noetherian induction and the smooth case, for any $X$ of dimension $\leq d$, we have that $M (X)^*(d)= \ihom_\Nis (M (X), \Z (d))\in \DM^\eff_\Nis$ is effective (\cf \cite[Lemma 8.6.1]{BVK}). Finally, recall that if $X$ is smooth and $d = \dim (X)$, the diagonal cycle 
$\Delta_X \in CH^d(X \times_K X)$  induces a canonical quasi-isomorphism (\cf \cite[Proof of Prop. 20.3 \& Ex. 20.11]{VL})
$$\Delta_{\Nis}\colon M^c (X) \longby{\sim} \ihom_\Nis (M (X), \Z (d))[2d]$$
yielding $M^c (X) \cong M (X)^*(d)[2d]\in \DM^\eff_\Nis$. In general, the Voevodsky motive $M(X)$, the motive with compact support $M^c(X)$, and the dual motive $M(X)^*(d)[2d]$ do not coincide if $X$ fails to be both smooth and proper. More precisely, if $X$ is smooth but not proper, Poincaré duality still yields an isomorphism $M^c(X) \cong M(X)^*(d)[2d]$ while these objects differ from $M(X)$. Conversely, if $X$ is proper but possesses singularities, we have $M(X) \cong M^c(X)$ but this complex is no longer isomorphic to the dual motive $M(X)^*(d)[2d]$. Under the change of topology monoidal functor $\alpha$, the same structural discrepancies naturally persist in $\DM^\eff_\et$.
\begin{lemma}[\protect{Duality}] \label{dual} For $X$ smooth of dimension $d = \dim (X)$ over $K= \bar K$ we have integrally defined duality isomorphisms which are compatible with the change of topology as shown in the following commutative diagram
\[\xymatrixcolsep{2pc}\xymatrix{ 
 H_{j,i}^{BM} (X)\ar[d]^{||}\ar[r]^{}& H_{j,i}^c (X)\ar[d]^{||}\\ 
H^{2d-j, d-i}_\Nis(X)\ar[r]^{}&  H^{2d-j,d-i} (X)
}\]
\end{lemma} 
\begin{proof} By adjunction (\cf \cite[Remarks 14.12 \&  8.21]{VL}) we have
$$\Hom_{\DM^\eff_{\tau}} (M (X)(i)[j],  \Z(d)[2d])\cong \Hom_{\DM^\eff_{\tau}} (\Z(i)[j], \ihom_\tau (M (X), \Z(d))[2d]) $$
and $M(X)^*(d) =\ihom_\tau (M (X), \Z (d))$ for both $\tau =\Nis $ and $\et$. 
 By composition with the isomorphism $\Delta_\Nis$ we obtain an isomorphism
$$\Hom_{\DM^\eff_{\Nis}} (\Z(i)[j], M^c (X)) \longby{\simeq} \Hom_{\DM^\eff_{\Nis}} (\Z(i)[j], \ihom_\Nis (M (X), \Z (d))[2d])$$
and $\alpha (M (X)^*(d))$ is $M (X)^*(d)$ in $\DM^\eff_{\et}$, \ie the canonical map 
$$\alpha\ihom_\Nis (M (X), \Z (d))\to \ihom_\et (M (X), \Z (d))$$
induced by the change of topology is an isomorphism in $\DM^\eff_{\et}$ over $K = \bar K$. Indeed, since $\alpha$ is a symmetric monoidal functor, it automatically maps duals to duals on the rigid subcategory of geometric motives. For the partial effective duality wrapped around the twist $\Z(d)$, the compatibility on internal $\ihom$ sheaves is immediate rationally because $\alpha_\Q\colon \DM_{\Nis,\Q} \to \DM_{\et,\Q}$ is an equivalence of tensor categories, while it holds on the torsion subgroups by Suslin rigidity, which guarantees that the finite-coefficient motivic complexes $\Z/m(d)$ coincide across both the Nisnevich and \'etale topologies (\cf \cite[Thm. 20.17]{VL}).
\end{proof}
\begin{lemma}[\protect{Poincar\'e Duality}]\label{pdual} For $X$ smooth over $K$, $\sigma: K \into \C$ and $\dim (X)= d $ we have 
$H^{2d-j, d-i}_\an(X)_\Q \cong H_{j, i}^{\an,c} (X)_\Q$, $H^{2d-j}_\dr(X) \cong H_{j}^{\dr,c} (X)$ and $H^{2d-j, d-i}_{\varpi}(X)_\Q \cong H_{j, i}^{\varpi,c} (X)_\Q$.
\end{lemma} 
\begin{proof}
For $q = d-i$ and $p = 2d -j$ we have 
$$H^{\an,c}_{j, i}(X)_\Q =\Hom_{\DM_{\et,\Q}} (\Z(i)[j], M^c(X)\otimes\beta^\sigma (\Q)) \cong \Hom_{\DM_{\et,\Q}} (M (X),\beta^\sigma \beta_\sigma \Q (q))[p])$$
by the duality isomorphism $M^c (X) \cong M (X)^*(d)[2d]$ and making use of the canonical quasi-isomorphism $\beta^\sigma \beta_\sigma \Q (q)= \Q (q)\otimes \beta^\sigma(\Q)$ in $\DM_{\et, \Q}\cong\DA_{\et, \Q}$ (\cf  the proof of Lemma \ref{compareBM}). Here $H^p(X_\an, \Q_\an(q))\cong \Hom_{\DM_{\et,\Q}} (M (X),\beta^\sigma \beta_\sigma \Q (q))[p])$ by adjunction. Moreover, we have 
$$H^{\dr,c}_{j}(X) =\Hom_{\DM_{\et}} (\Z[j], M^c(X)\otimes\bOm ) \cong \Hom_{\DM_{\et}} (M (X), \bOm (d)[2d-j])\cong H^{2d-j}_\dr(X).$$
Finally, the period isomorphism $\beta^\sigma\beta_\sigma \Z (q)\otimes_\Z \C \longby{\sim}\bOm \otimes_K \C$ is obtained as $\beta^\sigma\beta_\sigma \Z (q)\otimes_\Z \C \cong\Z (q)\otimes \beta^\sigma(\Z)\otimes_\Z \C\cong\bOm (q) \otimes_K \C\cong \bOm \otimes_K \C$
via $\varpi^{\Z(q)}$ in Lemma \ref{APLemma}.
Therefore, the period isomorphism in \cite[Lemma 1.2.1]{ABVB} and that in Lemma \ref{APLemma} for $M = M^c (X) \cong M (X)^*(d)[2d]$ corresponds under Poincar\'e duality and they provide the following
$$(H^{2d-j, d-i}_\an(X)_\Q, H^{2d-j}_\dr(X), \varpi_{M (X)}^{2d-j, d-i})\cong (H^{\an,c}_{j, i}(X)_\Q, H_{j}^{\dr,c} (X), \varpi_{j,i}^{M^c(X)}) $$
in the homological periods $\Q$-linear category (see \cite[\S 2.3]{ABVB} for details on period categories). We thus obtain the last claimed isomorphism $H^{2d-j, d-i}_\varpi(X)_\Q \cong H_{j, i}^{\varpi,c} (X)_\Q$. 
\end{proof} 
\begin{propose} \label{regdual}
Let $X$ be smooth and equidimensional over $K =\bar K$ and $\sigma : K \into \C$. For $0\leq i\leq d = \dim (X)$ we have a commutative square
\[\xymatrixcolsep{5pc}\xymatrix{ 
CH^{d-i}(X, j -2i)_\Q \ar[d]^{||}\ar[r]^{c\ell^{2d-j,d-i}_\varpi}& H^{2d-j,d-i}_\varpi (X)_\Q\ar[d]^{||}\\ 
CH_i(X, j -2i)_\Q\ar[r]^{c\ell_{j,i}^\varpi}&  H_{j,i}^{\varpi,c} (X)_\Q
}\]
where we identify period cohomology and homology by Poincar\'e duality. For $X$ proper and smooth $CH_0(X, j)= H_j^{Sus}(X)$ and the target $H_{j,0}^\varpi (X)$ is given by the usual singular homology.
\end{propose}
\begin{proof} By construction, Lemmas \ref{dual} - \ref{pdual} yield the following commutative square
\[\xymatrixcolsep{5pc}\xymatrix{ 
H^{2d-j, d-i}_\Nis(X) \ar[d]^{||}\ar[r]^{r^{2d-j, d-i}_{\varpi , \Nis}}& H^{2d-j,d-i}_\varpi (X)\ar[d]^{||}\\ 
H_{j,i}^{BM}(X)\ar[r]^{r^{\varpi,c}_{j,i}}&  H_{j,i}^{\varpi,c} (X)
}\]
which is the claimed one after the identification of the source motivic cohomology/homology with higher cycles. If $X$ is proper and $i=0$ we further have that $H_{j,i}^{BM}(X)=H_{j,0}^{Sus}(X)$.
\end{proof}

\section{Some computations}
\subsection{Motivic Albanese}
Recall the motivic Albanese triangulated functor (see \cite[Def. 5.2.1]{BVK} and \cite[Thm. 2.4.1]{BVA}) 
$$\LAlb  : \DM^\eff_{\gm}\to D^b (\M )$$ 
which is defined on the whole triangulated category $\DM^\eff_{\gm}\subset \DM^\eff_{\Nis}$ of effective geometric motives, taking values in the bounded derived category $D^b (\M )$ of $1$-motives.
Rationally, $\LAlb$ yields a left adjoint to the inclusion functor given by $$\Tot : D^b (\M )\to \DM^\eff_{\gm, \et}$$ where $\DM^\eff_{\gm, \et}\subset \DM^\eff_{\et}$ is the triangulated subcategory given by geometric motives under the change of topology \eqref{alpha} (see \cite[Def. 2.7.1]{BVK}). Integrally, recall that for any geometric motive $M\in \DM_\gm^\eff\subset \DM_{\Nis}^\eff$ we have a well defined map (see \cite[(5.1.2) \& Rem. 5.2.2]{BVK})
$$a_M \colon \alpha M\to \Tot \LAlb (M)$$
where $\alpha$ is the change of topology \eqref{alpha}.
By composition we obtain 
$$a\ell b_{j,i}^M: H_{j,i}(M)=\Hom_{\DM^\eff_\et}(\Z(i)[j], \alpha M)\to \Hom_{\DM^\eff_\et}(\Z(i)[j], \Tot \LAlb (M)):= H_{j,i}^{(1)}(M)$$
for any $M\in \DM_\gm^\eff$. Applying $\LAlb$ to the motive $M=M(X)$ of any algebraic scheme $X$ we obtain $\LAlb(X)$, for the motive with compact support $M = M^c(X)$ we get the Borel-Moore Albanese complex $\LAlb^c(X) \in D^b({}_t\M )$ (this is a complex of 1-motives with cotorsion, see \cite[Def. 8.5.1]{BVK}). Dually, we have $\RPic^c (X)\in D^b(\tM)$ (see \cite[\S 8.7]{BVK}). Finally, recall the cohomological Albanese complex $$\LAlb^*(X)\df \LAlb (M(X^{(d)})^*(d)[2d])\in D^b({}_t\M)$$ for $d = \dim (X)$ and where $X^{(d)}$ is the union of the $d$-dimensional components (see \cite[Def. 8.2.6]{BVK}). Dually, we have the homological Picard complex $\RPic^*  (X)\in D^b(\tM)$ (see \cite[\S 8.7]{BVK}).

In particular, for  $M = M(X)$ and $M = M^c(X)$ by composition we get 
$$H_{j,i}(X)=\Hom_{\DM^\eff_\et}(\Z(i)[j], M (X))\to \Hom_{\DM^\eff_\et}(\Z(i)[j], \Tot \LAlb (X)):= H_{j,i}^{(1)}(X)$$
and
$$H_{j,i}^c(X)=\Hom_{\DM^\eff_\et}(\Z(i)[j], M^c (X))\to \Hom_{\DM^\eff_\et}(\Z(i)[j], \Tot \LAlb^c (X)):= H_{j,i}^{(1),c}(X)$$
whence the change of topology map \eqref{cycle} induces the $i$-twisted motivic Albanese
$$a\ell b_{j,i}^c: CH_i(X, j-2i)\to H_{j,i}^{(1),c}(X)$$
for $X$ equidimensional. For $i=0$ we have $\Z (0) = \Tot [\Z \to 0]= \Z$ and we are in the situation of \cite[\S 13.6]{BVK}.\footnote{Note that the group $H_{j,0}^{(1),c}(X)$ is denoted $H^{c,(1)}_j(X)$ in \loccit} Since $\Tot$ is fully faithful we obtain the Albanese map from Suslin homology 
$a\ell b_{j}: H_j^{Sus}(X)\to H_{j,0}^{(1)}(X)= \HH^{-j}(K, \LAlb(X))$  and the Borel-Moore version
\begin{equation}\label{hcyclemap}
a\ell b_{j}^c: CH_0(X, j)\to H_{j,0}^{(1),c}(X)= \Hom_{D^b({}_t\M)} (\Z[j], \LAlb^c(X))= \HH^{-j}(K, \LAlb^c(X)).
\end{equation}
(Remark that for $i=1$ $\Z (1) \cong \Tot [0\to \G_m] = \G_m [-1]$ and we can also get the 1-twisted motivic Albanese 
$H_{j,1}^{(1), c}(X)= \EExt^{1-j} (\G_m, \LAlb^c(X))$
but we now just consider higher zero cycles). Similarly, for $M =M(X^{(d)})^*(d)[2d]$ we also get the map (see \cite[\S 13.7]{BVK})
\begin{equation}\label{cohalbmap}
a\ell b_{j}^*: H^{2d-j}_\eh(X, \Z (d))\to H^{2d-j,d}_{(1)}(X)=\HH^{-j}(K, \LAlb^*(X))
\end{equation}
where we may and will assume $X$ purely $d$-dimensional, for simplicity. 

In general, for any $M\in\DM^\eff_{\gm}$, by taking the $j$-th homology with respect to the cotorsion $t$-structue  (see \cite[Def. 1.13.1, Thm. 1.13.7 \& Notation 1.13.8]{BVK}) we let  $$\LAlb_j(M)\df {}_tH^{-j} (\LAlb (M))\in {}_t\M$$ the induced $1$-motives with cotorsion. In particular, we get $1$-motives with cotorsion 
$$\LA{j} (X)\df {}_tH^{-j} (\LAlb (X)),$$ as well as 
\begin{align*}\LA{j}^c (X)\df {}_tH^{-j} (\LAlb^c(X))\quad \text{and}\quad \LA{j}^* (X)\df {}_tH^{-j} (\LAlb^*(X))
\end{align*}
explicitely represented by 
$$\LA{j} (X)\df [L_j\by{u_j}G_j]\in {}_t\M$$ and the following 
$$\LA{j}^\dag(X)\df  [L_j^\dag\by{u_j^\dag}G_j^\dag]\in {}_t\M$$  for decorations $\dag= c$ or $*$.
Over $K=\bar K$ the abelian category ${}_t\M$ is of cohomological dimension $1$. Thus, for any $N\in{}_t\M$ and $M\in \DM^\eff_\gm$ we get short exact sequences
\begin{equation}\label{ext}
0\to \Ext (N, \LA{j+1}(M))\to \EExt^{-j} (N, \LAlb(M))
\longby{\pi_j}\Hom (N, \LA{j}(M))\to 0
\end{equation}
given by the canonical spectral sequence where the $\Hom$ and Yoneda $\Ext$ are in  ${}_t\M$  (\cf \cite[Proof of Lemma 3.2.1]{ABVB}). 
Recall that the $K$-points of a 1-motive $\LAlb_j(M)\in {}_t\M$ for $M\in \DM^\eff_\gm$ can be regarded as 
$$\LAlb_j(M)(K):= \Ext ([\Z\to 0], \LAlb_j(M)) \cong \Ext (\RPic^j(M), [0\to \G_m])$$
where $\RPic^j(M)$ is the Cartier dual of $\LAlb_j(M)$ and the isomorphism is a neat generalisation of the Weil-Barsotti formula (\cf \cite[Lemma 1.13.3, Prop. 1.13.5 \& Cor. 5.3.2]{BVK}).
\begin{lemma}\label{cohreproj} Let $X$ be over $K =\bar K$ of pure dimension $d = \dim (X)$ and assume that $\LA{k} (X),\LA{k}^\dag(X)\in \M$ are Deligne 1-motives for $k=j, j+1$. With or without decorations $\dag= c, *$ the group $\HH^{-j}(K, \LAlb^\dag(X))$ is an extension of the finitely generated free group $\ker (L_j^\dag\by{u_j^\dag}G_j^\dag(K))$ by the divisble group $$\LAlb_{j+1}^\dag(X)(K)=\coker (L_{j+1}^\dag\by{u_{j+1}^\dag}G_{j+1}^\dag(K))$$ and vanishes for $j<0$ or $j> \max (2, d+1)$.
The Albanese map $a\ell b_j$ and the decorations $a\ell b_j^c$ \eqref{hcyclemap} and  $a\ell b_j^*$ \eqref{cohalbmap} induce mappings
$$\delta_j: H_j^{Sus}(X)\to \ker (L_j\by{u_j}G_j(K)),$$
$$\delta_j^c: CH_0(X, j)\to \ker (L_j^c\by{u_j^c}G_j^c(K))$$
and
$$\delta_j^*: H^{2d-j}_\eh(X, \Z (d))\to \ker (L_j^*\by{u_j^*}G_j^*(K))$$
where $H_j^{Sus}(X)_j:= \ker \delta_j, CH_0(X, j)_j:=  \ker \delta^c_j$ and $H^{2d-j}_\eh(X, \Z (d))_j:= \ker \delta_j^*$ contain $D_j$, $D_j^c$ and $D^j$ the maximal divisible subgroups, respectively. For $X$ smooth $\delta_j^c\cong\delta_j^*$ under duality (see \S \ref{dualsect}) and for $X$ proper $\delta_j\cong\delta_j^c$. If $\LA{k} (X),\LA{k}^\dag(X)\in {}_t\M$ do have cotorsion the same holds true up to finite groups.
\end{lemma}
\begin{proof} 
Follows from $N=[\Z\to 0]$ and $M = M(X)$, $M = M^c(X)$ or $M=M(X^{(d)})^*(d)[2d]$ in the extension \eqref{ext}. We get the group $\Hom ([\Z\to 0], \LA{j}^\dag (X)) = \ker (L_j^\dag\by{u_j^\dag}G_j^\dag(K))$ which is free finitely generated, the divisible group 
$$\Ext ([\Z\to 0], \LA{j+1}^\dag(X))= \coker (L_{j+1}^\dag\by{u_{j+1}^\dag}G_{j+1}^\dag(K))$$
 and $\EExt^{-j} (\Z, \LAlb^\dag(X))\cong \HH^{-j}(K, \LAlb^\dag(X))$.  The maps $\delta_j^\dag$ are simply obtained by composition of $a\ell b_j$, $a\ell b_j^c$ or $a\ell b_j^*$  (from \eqref{hcyclemap} and \eqref{cohalbmap}) with $\pi_j$ and there is a factorisation through the maximal divisible subgroups $D_j$, $D_j^c$ and $D^j$.
\end{proof}
\begin{propose}[\protect{Higher Albanese mappings}]\label{highalb}
For any $M\in \DM^\eff_\gm$ over $K=\bar K$ we have higher Albanese mappings
$$a\ell b_{j}^M: H_{j}(M)_j\to \LAlb_{j+1}(M)(K)$$
where $H_{j}(M)_j$ is the kernel of the composition of $a\ell b_{j,0}^M$ with $\pi_j$ in \eqref{ext} for $N=[\Z\to 0]$.

In particular, for any algebraic $K$-scheme $X$, we then get: 
$$a\ell b_{j}: H_j^{Sus}(X)_j\to \LAlb_{j+1}(X)(K)$$ 
the Borel-Moore Albanese
$$a\ell b_{j}^c: CH_0(X, j)_j\to \LAlb_{j+1}^c(X)(K)$$
and the cohomological Albanese
$$a\ell b_{j}^*: H^{2d-j}_\eh(X, \Z (d))_j\to \LAlb_{j+1}^*(X)(K)$$
where  $\LAlb_{j+1}^\dag(X)(K)$ are the $K$-points as in Lemma \ref{cohreproj}.
\end{propose}
\begin{proof} Straightforward from  Lemma \ref{cohreproj} and \eqref{ext}.  \end{proof}
\subsection{Mixed realisations and de Rham-Betti formalism}\label{BdeR}
We can reconstruct the period regulators by making use of mixed and Betti-de Rham realisations.
Recall that we have the Betti-de Rham category $\Mod_{\Q,K}^{\cong}$ in \cite[\S 2 ]{ABVB} along with an enrichment given by a forgetful functor $\cMR\to \Mod_{\Q,K}^{\cong}$ from the category of polarizable mixed realisations for which we also have $\cMR_{(1)}\subset \cMR^\eff$ the subcategory of level 1-mixed realisations along with a left adjoint $( - )_{\leq 1}$ (see \cite[Prop. 16.1.2]{BVK} and  \cite[\S 6.3]{Hu}). Huber's mixed realisation $R^{\cMR}$ on $\DM_\gm^\Q$  is also providing a Betti-de Rham realisation $R_{\Q,K}$ (\cite[Remark 6.3.4 and Thm. 6.3.15]{Hu}) fitting in the following diagram
\[\xymatrixcolsep{5pc}\xymatrix{ 
\DM_\gm^\Q \ar@/^1.6pc/[rr]^{R_{\Q,K}} \ar@/^1.3pc/[d]^{\LAlb}\ar[r]^{R^{\cMR}}& D_{\cMR}\ar[r]^{\rm forget}\ar@/^1.3pc/[d]^{(\ )_{\leq 1}} &D^-_{\Q,K}\\ 
D(\cM_1^\Q)\ar[r]^{R^{\cMR}_1}\ar@/^1.3pc/[u]^{\Tot}&  D_{\cMR ,{(1)}}\ar@/^1.3pc/[u] &}
\]
where $D^-_{\Q,K}$ is the derived category of homological $(\Q,K)$-vector spaces (akin to \cite[\S 5.3]{Hu}). We have Betti-de Rham cohomology
$$H_{j}R_{\Q,K}(M(X)^*(d)[2d]) = (H^{2d-j, d}_\an(X)_\Q, H^{2d-j}_\dr(X), \varpi_{M (X)}^{2d-j, d}):= H_{\rm BdR}^{2d-j, d}(X)_\Q,$$ 
the Borel-Moore Betti-de Rham homology
$$H_jR_{\Q,K}(M^c(X)) = (H^{\an,c}_{j, 0}(X)_\Q, H_{j}^{\dr,c} (X), \varpi_{j,0}^{M^c(X)}):=H^{\rm BdR, c}_{j, 0}(X)_\Q$$
and the same holds true for the ordinary homology. Moreover
$$\Hom (\Z, H_{j}R_{\Q,K}(M(X)^*(d)[2d]) )= H^{2d-j, d}_\varpi (X)_\Q$$
and 
$$\Hom (\Z, H_jR_{\Q,K}(M^c(X)) )= H_{j, 0}^{\varpi,c}(X)_\Q$$
where the $\Hom$ here is taken in $\Mod_{\Q,K}^{\cong}$ (see  \cite[Cor. 3.1.2]{ABVB}). 

Here, the constant object $\mathbb{Z}$ is viewed in the respective context as follows. Within the category of homological periods $\mathrm{Mod}_{\mathbb{Q},K}^{\cong}$ (see \cite[Def. 2.4.1]{ABVB}), $\mathbb{Z}$ represents the trivial period object $(\mathbb{Q}, K, \text{id})$, where the Betti lattice is rationalized and paired via the identity comparison with the base field. 

We also have a Betti-de Rham realisation for 1-motives 
$$T_{\rm BdR}^\Q:\cM_1^\Q \to\Mod_{\Q,K}^{\cong}$$
which is exact and faithful (see \cite[Def. 2.5.1]{ABVB}) and 
$$T_{\rm BdR}^\Q:D^b(\cM_1^\Q) \to D^-_{\Q, K}$$
whose mixed enrichement is given by $R_1^\cMR$ (as a consequence of the compatibility stated in \cite[Thm. 15.4.1]{ABVB} and \cite{VO}). 
\begin{lemma} \label{key}
For any $M\in\DM_\gm^{\eff,\Q}$ we have isomorphisms
$$R_{\Q,K}(M)_{\leq 1}\iso T_{\rm BdR}^\Q(\LAlb(M)) \hspace{0.5cm}\text{and}\hspace{0.5cm} H_j(R_{\Q,K}(M))_{\leq 1}\iso T_{\rm BdR}^\Q(\LA{j}(M))$$
in $D^-_{\Q, K}$ and $\Mod_{\Q,K}^{\cong}$, respectively, where $(-)_{\leq 1}$ is given by forgetting. These isomorphisms induce the following commutative square
\[\xymatrixcolsep{1.8pc}\xymatrix{ 
H_{j, 0}(M)_\Q \ar[d]^{a\ell b_{j,0}^M}\ar[r]^{r_{j,0}^\varpi}& H_{j, 0}^\varpi(M)_\Q    \ar[d]^{( \ )_{\leq 1}}\\ 
H_{j, 0}^{(1)}(M)_\Q\ar[r]^-{r_{j,0}^{\varpi, (1)}}&  T_{\varpi}^\Q(\LA{j} (M))
}\]
where $r_{j,0}^{\varpi, (1)}$ is surjective over $K=\bar \Q$.
\end{lemma}
\begin{proof} The claimed isomorphisms are an instance of \cite[Thm. 16.3.1 \& Cor. 16.3.2]{BVK} in $D_\cMR$ and $\cMR$ thus also in $D^-_{\Q, K}$ and $\Mod_{\Q,K}^{\cong}$ by forgetting. The claimed square is the following one
\[\xymatrixcolsep{1.8pc}\xymatrix{ 
\Hom_{\DM_{\gm}^{\eff,\Q}} (\Z[j] , M) \ar[d]^{a\ell b_{j,0}^M}\ar[r]^{r_{j,0}^\varpi}& \Hom_{\Mod_{\Q,K}^{\cong}} (\Z, H_jR_{\Q,K}(M))   \ar[d]^{( \ )_{\leq 1}}\\ 
\Hom_{D^b(\M^\Q)} (\Z[j], \LAlb (M))\ar[r]^-{r_{j,0}^{\varpi, (1)}}&  \Hom_{\Mod_{\Q,K}^{\cong}} (\Z , T_{\rm BdR}^\Q(\LA{j} (M)))
}\]
which is obtained from the commutative square 
    \[\xymatrixcolsep{5pc}\xymatrix{ 
\Hom_{\DM_{\gm}^{\eff,\Q}} (\Z[j] , M) \ar[d]^{\LAlb}\ar[r]^{R_{\Q,K}}& \Hom_{D^-_{\Q,K}} (\Z[j], R_{\Q,K}(M))  \ar[d]^{( \ )_{\leq 1}}\\ 
\Hom_{D^b(\M^\Q)} (\Z[j] , \LAlb (M))\ar[r]^{T^\Q_{\rm BdR}}&  \Hom_{D^-_{\Q,K}} (\Z[j], T_{\rm BdR}^\Q(\LAlb (M)))
}\]
passing to $H_j$ where the composite
\[
\xymatrixcolsep{1.2pc}\xymatrixrowsep{1.2pc}\xymatrix{
\Hom_{D^b(\M^\Q)} (\Z[j] , \LAlb (M)) \ar[rr]^-{r_{j,0}^{\varpi, (1)}} \ar[dr]_-{\pi} & & \Hom_{\Mod_{\Q,K}^{\cong}} (\Z , T_{\rm BdR}^\Q(\LA{j} (M))) \\
& \Hom_{\M^\Q} (\Z, \LA{j} (M)) \ar[ur]^-{\sim}_{T^\Q_{\rm BdR}} &
}
\]
is onto: actually, over $K=\bar \Q$, $\pi$ in \eqref{ext} for $N=\Z$ is onto and the faithful functor $T_{\rm BdR}^\Q$ is also full by \cite[Thm. 2.7.1]{ABVB}. When evaluating extensions in the motivic framework, $\mathbb{Z}$ represents the base Tate motive $\mathbb{Z}(0) \cong \Z[0]$ viewed inside the rationalized effective category $\DM^{\eff,\mathbb{Q}}_{\rm gm}$ via the canonical $\Tot$ functor. Here, for a motive $M \in \DM_\gm^{\eff,\Q}$, the homological group $H_{j,0}(M)_\Q$ is defined according to \cite[Chapter 6]{BVK}. The compatibility of the realization functors ensures that the top horizontal arrow in the structural diagram matches exactly the period regulator $r_{j,0}^\varpi$. To see this explicitly, we recall the foundational properties of Huber's realization functor $R_{\mathbb{Q},K} \colon \DM_{\gm}^{\eff,\Q} \to D_{\cMR}$ (\cf \cite[Thm. 6.3.15]{Hu}). This functor is a symmetric monoidal triangulated functor which sends the constant unit motive $\mathbb{Q}$ to the trivial polarizable period object $\mathbf{1} \df (\mathbb{Q}, K, \text{id}) \in \Mod_{\mathbb{Q},K}^{\cong}$, and maps any geometric motive $M(X)_\mathbb{Q}$ to its total Betti-de Rham complex. Since $R_{\mathbb{Q},K}$ is constructed by integrating the stable Betti and de Rham structural realizations alongside period's comparison isomorphism in Lemma \ref{APLemma}, the application of the homological functor $H_j$ to the monoidal unit morphism naturally commutes with the period boundaries. Therefore, the induced mapping on the spaces of sections coincides with the canonical intersection of the Betti and de Rham lattices, yielding precisely the defined homological period regulator $r^\varpi_{j,0}$. (Note that the named compatibilties are homological analogue to \cite[Lemmas 3.2.1-3.2.3]{ABVB}. See also \cite[Thm. 15.4.1]{BVK} for the compatibility of Deligne's Betti realization on $1$-motives and the Betti realization $\beta_\sigma$.)
\end{proof}
\begin{thm} \label{1-motivic}
With rational coefficients there exist \emph{surjective} maps 
$r_{j,0}^{\varpi, (1)}:H_{j,0}^{(1)}(X)_\Q\to  T_\omega^\Q (\LA{j}(X))$,
$c\ell_{j,0}^{\varpi, (1)}:H_{j,0}^{(1), c}(X)_\Q\to  T_\omega^\Q(\LA{j}^c(X))$ and 
$r^{2d-j, d}_{\varpi, (1)}:H^{2d-j}_{\eh, (1)}(X, \Q (d))\to T_\omega^\Q (\LA{j}^*(X))$ induced by the $1$-motivic part of the period homology and cohomology, respectively, for all $j\geq 0$.
\end{thm}
\begin{proof}
    It follows from Lemma \ref{key} for $M= M(X), M^c(X)$ or $M=M(X)^*(d)[2d]$.
\end{proof}
\subsection{Case j=0}
We now focus on the case $j=0$. For any algebraic scheme $X$ over $K = \bar K$, the zero-th Suslin homology group $H^{Sus}_0(X) \df \Hom_{\DM^\eff_{\gm}}(\mathbb{Z}, M(X))$ represents the classical abelian group of zero-cycles on $X$ modulo Suslin equivalence (\cf \cite{VL}). In this setting, the homological period regulator mapping reduces to 
$$r^\varpi_{0,0} \colon H^{Sus}_0(X)_\Q \longrightarrow H^\varpi_{0,0}(X)_\Q,$$
which is the classical period mapping. By specializing Lemma \ref{key} to $M = M(X)_\mathbb{Q}$, this map factors through the rationalized Albanese variety. 

Denote by $\pi_0^c(X)$ the disjoint union of $\pi_0(X_i)$ where $X_i$ runs through the proper connected components of $X$: this is also called the scheme of constants with compact support (see \cite[Def. 10.6.1]{BVK}). 

Let $X$ be equidimensional of dimension $d$. If $X$ is not smooth, let $p: \tilde X\to X$ be an abstract blow-up of $S\subset X$ such that $\tilde X$ is smooth and $\tilde S = p^{-1}(S)$ (see \cite[Def. 12.21]{VL}, \cf \cite[\S 8.2.6]{BVK}). 
We have:
\begin{lemma} \label{00}
Let $X$ be over $K =\bar \Q$ and $d =\dim (X)$. 
 We have that
  $$H_{0, 0}^\varpi(X)\cong \LA{0}(X)\cong \Z[\pi_0(X)],$$
 $$H_{0, 0}^{\varpi,c}(X)\cong \LA{0}^c(X)\cong \Z[\pi_0^c(X)] $$
 and
 $$H^{2d, d}_\varpi(X) \cong \LA{0}^*(X)\cong \Z[\pi_0^c(\tilde X)]\cong H^{2d, d}_\varpi(\tilde X)$$ where $\tilde X$ is any resolution.
\end{lemma}
\begin{proof} For $j=0$ we have that $H_{0, 0}^{\rm BdR, \dag}(X)\cong T_{\rm BdR}\LA{0}^\dag(X)$ and 
$H^{2d, d}_{\rm BdR}(X) \cong T_{\rm BdR}\LA{0}^*(X)$. Since  $\LA{0}(X)= [\Z[\pi_0(X)] \to 0]$, $\LA{0}^c(X)= [\Z[\pi_0^c(X)] \to 0]$ ( \cite[Prop. 10.4.2 b) \& 10.6.2 b)]{BVK}) and $\LA{0}^*(X)= [\Z[\pi_0^c(\tilde X)] \to 0]$ (same computation as for \cite[Lemma 12.12.1 \& Thm. 12.12.6]{BVK}) are Artin motives we get all the claims. Just remark that $H^{2d, d}_{\rm BdR}(X) \cong H^{2d, d}_{\rm BdR}(\tilde X)$. For $X$ smooth $M(X)^*(d)[2d]\cong M^c(X)$ and we have that 
$$H^{2d, d}_\varpi(X) \cong H_{0,0}^\varpi(X)=\Z[\pi_0^c(X)]$$
from Lemma \ref{pdual}.
\end{proof}
\begin{lemma}  \label{blow}
For $p: \tilde X\to X$ an abstract blow-up of $S\subset X$ such that $X$ is reduced, $\tilde X$ is smooth and $\tilde S = p^{-1}(S)$ we have an exact sequence
 $$H^{2d-1}_{\cdh}(\tilde S, \Z (d))\to H^{2d}_\cdh(X, \Z (d)) \longby{p^*} CH_0(\tilde X)\to 0$$ 
 for $d=\dim (X)$.
\end{lemma}
\begin{proof}
   In fact, we have the following long exact sequence
\[\xymatrix{ 
H^{2d-1}_{\cdh}(\tilde S, \Z (d))\ar[r]& H^{2d}_\cdh(X, \Z (d))\ar[r]  & H^{2d}_\cdh(\tilde X, \Z (d))\oplus H^{2d}_\cdh(S, \Z (d))\ar[r] & H^{2d}_\eh(\tilde S, \Z (d))}\]
where $\dim S, \dim \tilde S \leq d-1$ and just recall that $H^{p}_\cdh(Z, \Z (q))=0$ if $p> \dim Z + q$ by cdh-descent, \ie by blow-up induction from the smooth case.\footnote{Thanks to A. Krishna for suggesting to use descent.}
\end{proof}
\begin{thm} \label{zero}
For any $X$ over $K=\bar \Q$ the period regulators $r_{0, 0}^{\varpi}:H_0^{Sus}(X)\to H_{0, 0}^\varpi(X)$, $c\ell_{0,0}^{\varpi}: CH_{0}(X)\to  H_{0, 0}^{\varpi,c}(X)$ and 
$r^{2d, d}_{\varpi}:H^{2d}_{\eh}(X, \Z (d))\to H^{2d, d}_\varpi(X)$ 
are surjections for $d=\dim (X)$.
\end{thm}
\begin{proof}
    It follows from the smooth case and Lemma \ref{00} and \ref{blow}.
\end{proof}
\begin{remark}\label{zerork}
    Note that the mappings in Theorem \ref{zero} are algebraically defined via Lemmas \ref{cohreproj} - \ref{00} and surjectives over any algebraically closed field.
\end{remark}

\subsection{Case j=1}
Consider $\LA{1}(X) = [L_1\by{u_1}G_1]$. We have that $\LA{1}(X)(K)$ is given by the $K$-points of the $\eh$-Albanese scheme as explained in 
\cite[Prop. 12.6.5 b)]{BVK}. For $X$ normal we have that $L_1=0$ and $G_1$ is the usual Serre-Albanese semiabelian scheme (see  \cite[\S 12.6]{BVK}).

For $\LA{1}^c(X) = [L_1^c\by{u_1^c}G_1^c]$ recall that the 1-motive $\LA{1}^c(X)$ is Cartier dual to 
$$\RA{1}^c(X)= [0\to \Pic^0(\bar X, Z)/\cU]$$ by \cite[Thm. 12.11.1]{BVK} for
$\bar X$ a compactification with closed complement $Z$, and where $\Pic_{(\bar X, Z)/K}$ is the relative Picard $K$-scheme and $\cU$ its unipotent radical. For $X$ proper we have that $\RA{1}^c(X)=\RA{1}(X)= [0\to \Pic^0(X)/\cU]$ is the semi-abelian variety given by the simplicial Picard functor as described in \cite[\S 4.1]{BVS}, the map $u_1^c$ is explicitely described in Prop. 5.1.4 \loccit and, in particular: if $X$ is proper and normal then $L_1^c=0$, $\Pic^0(X)$ is an abelian variety and $G_1^c=\Pic^0(X)^\vee$. For $X$ normal, connected not proper with a normal compactification $\bar X$ and $Z=\bar X-X$ we have that the rank of $L_1^c$ is $\# \pi_0(Z)-1$ by \cite[Cor. 12.11.2]{BVK}.

Moreover, we have $\LA{1}^*(X) = [L_1^*\by{u_1^*}G_1^*]$ which is the Cartier dual of $$\RA{1}^* (X) \cong [\Div_{\bar S/S}^0(\bar X, Y)\by{u_1} \Pic^0{(\bar X, Y)}]$$ by \cite[Thm. 12.12.6]{BVK} where $S$ is the singular locus of $X$, $\bar X$ a normal crossing smooth compactification of $\tilde X$ a resolution of singularities of $X$ with boundary divisor $Y$, $\bar S$ the Zariski closure of the reduced inverse image of $S$ in such a way that $\bar S + Y$ is a reduced normal corossing divisor on $\bar X$, and $\Div_{\bar S/S}^0(\bar X, Y)$ is the group of divisors on $\bar X$ which have support on $\bar S$ disjoint from $Y$, trivial push-forward on $S$ and algebraically equivalent to zero relative to $Y$ (as described in \cite[\S 2.1 - 3.1]{BVS}). For $X$ is proper then $L_1^*=0$ and $\LA{1}^*(X)$ is a semi-abelian variety whose abelian quotient is $\Alb (\tilde X)$. 

If $X$ is smooth 
$\LA{1}^c(X) =\LA{1}^*(X)$
and if it is not proper it is explicitely described in \cite[Prop. 3.1.4]{BVS} as follows:
$L_1^c=L_1^* = \ker \Z[\pi_0(Y)]\by{\gamma} \Z[\pi_0(\bar X)]]$
where $\gamma$ is induced by the mapping that takes a component of $Y$ to the component
of $\bar X$ to which it belongs, $G_1^c=G_1^* = \coker \oplus_i\Alb(Y_i)\to \Alb (\bar X)$
for $Y= \cup Y_i$ a decomposition of $Y$ in irreducioble components and $u_1^c=u_1^*$ is given by $a (y_i-y_j)\in \Alb (\bar X)$ for a choice of (closed) points $y_i\in Y_i$ and $y_j \in Y_j$, where $Y_i$ and $Y_j$ are distinct connected components of $Y$, contained in the same component of $\bar X$, where $a : \cZ^0(X)^0 \to \Alb (\bar X)$ denotes the Albanese mapping for zero-cycles of degree zero.
\begin{lemma} \label{11}
Let $X$ be over $K =\bar \Q$ and $d =\dim (X)$. 
We have that
 $$\ker u_1\cong H_1(X_\an, \Z)_\fr\cap H_1^\dr(X)=H_{1,0}^{\varpi}(X)_\fr,$$
 $$\ker u^c_1\cong H_1^{BM}(X_\an, \Z)_\fr\cap H_1^{\dr,c}(X)=H_{1,0}^{\varpi,c}(X)_\fr $$
 and
$$\ker u^*_1 \cong H^{2d-1}(X_\an, \Z_\an (d))_\fr \cap H^{2d-1}_\dr(X) =H^{2d-1,d}_\varpi (X)_\fr$$ 
where $\fr$ is the free part.
\end{lemma}
\begin{proof} With or without decorations $\dag= c, *$ we have 
$$\ker u^\dag_1=\Hom (\Z, \LA{1}^\dag (X))\iso \Hom (\Z , T_{\rm BdR}(\LA{1}^\dag (X)))=T_\varpi(\LA{1}^\dag (X))$$
where $T_{\rm BdR}$ is fully faithful by \cite[Thm. 2.7.1]{ABVB}, $T_\varpi(\LA{1} (X))= H_{1,0}^\varpi(X)_\fr$, $T_\varpi(\LA{1}^c (X))= H_{1,0}^{\varpi,c}(X)_\fr$ and, finally,  $T_\varpi(\LA{1}^* (X))=H^{2d-1,d}_\varpi (X)_\fr$. 
\end{proof}
\begin{thm} \label{one}
Let $X$ be over $K=\bar \Q$. We have that $H^{2d-1}(X_\an, \Z_\an (d))_\fr\cap H^{2d-1}_\dr(X)=0$ if $X$ is proper and
$H_1(X_\an, \Z)_\fr\cap H_1^\dr(X)=0$ if $X$ is normal.
 
In general, we have an injection
$$r_\varpi^{2d-1,d}\otimes \Q/\Z: H^{2d-1}_\cdh(X, \Z (d))\otimes\Q/\Z\to  H^{2d-1,d}_\varpi (X)\otimes\Q/\Z$$
which is also surjective if and only if the surjection  
$$a\ell b_{0,\tors}^*: H^{2d}_\cdh(X,\Z (d))_\tors\to \LA{1}^*(X)(K)_\tors$$ 
induced by $a\ell b_{0}^*$ in Proposition \ref{highalb} is an injection. Moreover, we get that
$$c\ell^\varpi_{1,0}\otimes \Q/\Z:CH_0(X, 1)\otimes\Q/\Z\to H_{1, 0}^{\varpi,c}(X)\otimes\Q/\Z$$
is an injection which is also surjective if and only if the surjection  
$$a\ell b_{0,\tors}^c: CH_0(X)_\tors\to \LA{1}^c(X)(K)_\tors$$ induced by
$a\ell b_{0}^c$ in  Proposition \ref{highalb} is an injection. Finally, 
$$r^\varpi_{1,0}\otimes \Q/\Z:H_0^{Sus}(X)\otimes\Q/\Z\to H_{1, 0}^{\varpi}(X)\otimes\Q/\Z$$
is an injection which is also surjective if and only if the surjection  
$$a\ell b_{0,\tors}: H_0^{Sus}(X)_\tors\to \LA{1}(X)(K)_\tors$$ induced by
$a\ell b_{0}$ in  Proposition \ref{highalb} is an injection.
\end{thm}
\begin{proof} The claimed vanishings are consequences of Lemma \ref{11}. Consider the following commutative diagram with exact rows, induced by the change of coefficients: 
\[\xymatrix{ 
0\to H^{2d-1}_\cdh(X, \Z (d))\otimes\Q/\Z\ar[d]_{\delta^*_1\otimes \Q/\Z}\ar[r]& H^{2d-1}_\cdh(X,\Q/\Z (d))\ar[d]^{||}\ar[r]& 
H^{2d}_\cdh(X,\Z (d))_\tors\to 0\ar[d]^{a\ell b_{0,\tors}^*}\\ 
0 \to H^{2d-1}_{(1)}(X, \Z (d))\otimes\Q/\Z\ar[r]& H^{2d-1}_{(1)}(X, \Q/\Z (d))\ar[r]& H^{2d}_{(1)}(X, \Z (d))_\tors \to 0
}\]
where: $H^{2d-1}_{(1)}(X, \Z (d))\otimes\Q/\Z\cong \ker u^*_1\otimes\Q/\Z$ and  $H^{2d}_{(1)}(X, \Z (d))_\tors\cong (\coker u_1^*)_\tors$ by Lemma \ref{cohreproj}, $\delta^*_1\otimes \Q/\Z = r_\varpi^{2d-1,d}\otimes \Q/\Z$ by Lemmas \ref{11}-\ref{key}, the middle vertical is an isomorphism by \cite[Thm. 13.7.3]{BVK}. The other claims are then clear by diagram chase. Similarly, we also have:
\[\xymatrix{ 
0\to CH_0(X, 1)\otimes\Q/\Z\ar[d]_{\delta_1^c\otimes \Q/\Z}\ar[r]& H_{1}(X, \Q/\Z)\ar[d]^{||}\ar[r]& 
CH_0(X)_\tors\to 0\ar[d]^{a\ell b_{0,\tors}^c}\\ 
0 \to H_{1,0}^{(1),c}(X, \Z)\otimes\Q/\Z\ar[r]& H_{1,0}^{(1),c}(X, \Q/\Z)\ar[r]& H_{0,0}^{(1),c}(X, \Z)_\tors \to 0
}\]
where: $H_{1,0}^{(1),c}(X, \Z)\otimes\Q/\Z\cong \ker u^c_1\otimes\Q/\Z$ and  $H_{0,0}^{(1),c}(X, \Z (d))_\tors\cong (\coker u_1^c)_\tors$, $\delta^c_1\otimes \Q/\Z = c\ell^\varpi_{1,0}\otimes \Q/\Z$ by the same  Lemmas \ref{cohreproj}-\ref{11}-\ref{key}, and the middle vertical is an isomorphism by \cite[Prop. 13.6.1]{BVK}. Finally, for Suslin homology, the same arguments apply using \cite[Cor. 13.4.4]{BVK}.
\end{proof}
Over $K=\bar K$ the snake Lemma applied to the diagrams in the proof of Theorem \ref{one} is providing the connecting homomorphisms 
\begin{align*}
 \partial^*&:\ker (a\ell b_{0,\tors}^*) \longby{\sim} \coker\left(\delta_1^*\otimes \Q/\Z\right), \\
 \partial^c&:\ker (a\ell b_{0,\tors}^c) \longby{\sim} \coker\left(\delta_1^c\otimes \Q/\Z\right) \quad \text{and} \\
 \partial^{Sus}&: \ker (a\ell b_{0,\tors}) \longby{\sim} \coker\left(\delta_1\otimes \Q/\Z\right)
\end{align*}
as already stated in \cite[Cor. 13.6.2, 13.7.4 \& 13.5.2]{BVK}. Over $K=\bar \Q$ the period conjecture predicts the vanishing of these connecting homomorphisms: we then may optimistically expect that this holds true over an arbitrary $K=\bar K$.

\begin{coreconj}[Generalized Ro\u{\i}tman] \label{coreq}
Let $X$ be an algebraic scheme of pure dimension $d$ which is defined over an algebraically closed field $K = \bar K$. Then the connecting homomorphisms $\partial^*, \partial^c$, and $\partial^{Sus}$ are \textit{zero}. Equivalently:
\begin{points}
\item The generalized \textit{Cohomological Ro\u{\i}tman Theorem} holds true, \ie the cohomological Albanese map on torsion subgroups 
$$a\ell b_{0,\tors}^*: H^{2d}_\cdh(X,\Z (d))_\tors \longby{\sim} \LA{1}^*(X)(K)_\tors$$
is an isomorphism.
\item The generalized \textit{Chow--Ro\u{\i}tman Theorem} holds true, \ie the Borel--Moore Albanese map on torsion subgroups
$$a\ell b_{0,\tors}^c: CH_0(X)_\tors \longby{\sim}  \LA{1}^c(X)(K)_\tors$$
is an isomorphism.
\item The generalized \textit{Suslin--Ro\u{\i}tman Theorem} holds true, \ie the plain Albanese map on torsion subgroups
$$a\ell b_{0,\tors}: H_0^{Sus}(X)_\tors \longby{\sim} \LA{1}(X)(K)_\tors$$
is an isomorphism.
\end{points}
Furthermore, for any degree $j \geq 1$, the higher motivic Albanese mappings introduced in Proposition \ref{highalb}
$$a\ell b_{j}\colon H_j^{Sus}(X)_j \to \LAlb_{j+1}(X)(K), \quad a\ell b_{j}^c\colon CH_0(X, j)_j \to \LAlb_{j+1}^c(X)(K)$$
and $a\ell b_{j}^*\colon H^{2d-j}_\eh(X, \Z (d))_j \to \LAlb_{j+1}^*(X)(K)$ are \textit{surjections on torsion subgroups}.
\end{coreconj}

\begin{remark}\label{onerk}
For $X$ smooth and proper we have that $a\ell b_{j}=a\ell b_{j}^c=a\ell b_{j}^*$ in Proposition \ref{highalb}, $\LAlb_{0}(X)(K)= \Z[\pi_0(X)]$, $\LAlb_{1}(X)(K)= \Alb(X)(K)$ are $K$-points of the Albanese variety, $\LAlb_{2}(X)(K)= \NS(X)^*$ (modulo Enriques torsion) is the torus given by the N\'eron-Severi and $\LAlb_{j}(X)(K)=0$ for $j\geq 3$. Therefore
$a\ell b_{0}:CH_0(X)_0\to \Alb(X)(K)$ is the classical Albanese map for zero cycles which have degree $0$ on each component and the map
$a\ell b_{1}:CH_0(X,1)\to \NS(X)^*$ is dual of the classical cycle map. In fact, we have $CH_0(X,j)\cong H^{2d-j}_\et(X, \Z (d))$ and $H^{2d-j}_\et(X, \Z (d))\otimes \Q/\Z=0$ for $j\neq 0$ (by \cite[Thm. A.13]{ABVB}) whence 
$H_{j+1}(X,\Q/\Z)\cong CH_0(X,j)_\tors$ for all $j$. We have that $H_{j+1}(X,\Q/\Z)\to H_{j+1}^{(1)}(X,\Q/\Z)$ is an iso for $j=0$ and for $j=1$ is a surjection (by \cite[Cor. 13.4.4]{BVK}). Thus we get the classical Ro\u{\i}tman theorem that $a\ell b_{0,\tors}$ is an isomorphism and the surjection
$a\ell b_{1,\tors}:CH_0(X,1)_\tors\to \NS(X)^*_\tors$ which is dual to the cycle map $\NS (X)\otimes \Q/\Z\into H^2_\et(X,\Q/\Z(1))$ with cokernel ${\rm Br}(X)_\tors$ (a refinement of \cite[Thm. 13.4.5 c) \& Rem. 13.4.6]{BVK}).
\end{remark}

\subsubsection*{Acknowledgements} 
I would like to thank the IH\'ES (Paris, France) and the TIFR (Mumbai, India) for providing excellent working conditions and the generous support that enabled significant advancements on the subject matter of this paper. I am also grateful to the MCM (Beijing, China) for their warm hospitality and for the opportunity to deliver a seminar on the status of the period conjectures. Finally, I warmly thank the referee for the careful reading and numerous insightful comments that have considerably improved the exposition of this manuscript.

\end{document}